\documentclass[11 pt]{amsart}

\usepackage{latexsym}
\usepackage{amssymb, amsmath}
\usepackage{amsthm}
\usepackage{geometry}
\usepackage[pdftex,colorlinks]
{hyperref}

\numberwithin{equation}{section}

\newtheorem{theorem}{Theorem}[section]
\newtheorem{lemma}[theorem]{Lemma}

\newtheorem{proposition}[theorem]{Proposition}

\newtheorem{defn}[theorem]{Definition}

\newtheorem{remark}[theorem]{Remark}

\def\beq{\begin{equation}}
\def\eeq{\end{equation}}

\def\cA{\mathcal{A}}
\def\cB{\mathcal{B}}\def\B{\mathcal{B}}
\def\cC{\mathcal{C}}
\newcommand{\C}{\mathcal{C}}

\def\cE{\mathcal{E}}
\def\cF{\mathcal{F}}
\newcommand{\F}{\mathcal{F}}
\def\G{\mathcal{G}}

\def\cI{\mathcal{I}}\def\I{\mathcal{I}}
\def\eps{\varepsilon}

\newcommand{\Lp}{\mathcal{L}}
\def\cM{\mathcal{M}}
\def\cN{\mathcal{N}}
\def\cO{\mathcal{O}}
\def\cP{\mathcal{P}}

\def\cR{\mathcal{R}}
\def\cS{\mathcal{S}}
\newcommand{\Si}{\mathcal{S}}

\def\cW{\mathcal{W}}
\newcommand{\W}{\mathcal{W}}

\newcommand{\N}{\mathbb{N}}

\newcommand{\Ho}{\mathbb{H}}\newcommand{\bH}{\mathbb{H}}

\newcommand{\bn}{\mathbf{n}}

\newcommand{\p}{\mathbf{p}}
\newcommand{\bp}{\mathbf{p}}
\newcommand{\bq}{\mathbf{q}}
\newcommand{\bv}{\mathbf{v}}
\newcommand{\bx}{\mathbf{x}}
\newcommand{\bE}{\mathbf{E}}
\newcommand{\bF}{\mathbf{F}}
\newcommand{\bG}{\mathbf{G}}
\newcommand{\bP}{\mathbf{P}}

\newcommand{\sign}{\text{sign}}

\newcommand{\pd}{\partial}
\newcommand{\vf}{\varphi}
\newcommand{\po}{\psi_1}
\newcommand{\pt}{\psi_2}
\newcommand{\bpsi}{\overline{\psi}}
\newcommand{\R}{\mathbb{R}}
\newcommand{\ve}{\varepsilon}

\newcommand{\ds}{\displaystyle}

\newcommand{\hmu}{\hat{\mu}}
\newcommand{\hg}{\hat{g}}

\begin{document}
\title[Fluctuation of the entropy production for the Lorentz gas under small external forces]
{Fluctuation of the entropy production for the Lorentz gas under small external forces}
\author{Mark F. Demers, Luc Rey-Bellet \and Hong-Kun Zhang}\address{Mark F. Demers, Department of Mathematics,
Fairfield University, Fairfield CT 06824, USA.}
\email{mdemers@fairfield.edu}
\address{Luc Rey-Bellet, Department of
Mathematics and Statistics, UMass Amherst, MA 01003, USA.}
\email{luc@math.umass.edu}
\address{Hong-Kun Zhang, Department of
Mathematics and Statistics, UMass Amherst, MA 01003, USA.}
\email{hongkun@math.umass.edu}

\thanks{M.\ D.\ is partially supported by NSF Grant DMS-1362420; L.\ R.-B.\ is partially supported by NSF Grant DMS-1515712.; H.-K.\
Z.\ is partially supported by NSF  Grant DMS-1151762.}

\date{\today}

\begin{abstract}  In this paper we study the physical and statistical properties of the periodic Lorentz 
gas with finite horizon driven to a non-equilibrium steady state by the combination of  
non-conservative external forces and deterministic thermostats.  A version of this model 
was introduced by Chernov, Eyink, Lebowitz, and Sinai and subsequently generalized by 
Chernov and the third author.   Non-equilibrium steady states for these models are SRB measures 
and they are  characterized by the positivity of the steady state entropy production rate.  
Our main result is  to establish that the entropy production, in this context equal to the phase space 
contraction,  satisfies the Gallavotti-Cohen fluctuation relation.  The main tool needed in the proof 
is the family of anisotropic Banach spaces introduced by the first and third authors to study the ergodic and 
statistical properties of billiards using transfer operator techniques.  
\end{abstract}

\maketitle

\section{Introduction}\label{sec:intro}

The periodic Lorentz gas (or Sinai billiard) is obtained by placing finitely many disjoint scatterers
with smooth boundaries of strictly positive curvature on the 2-torus. 
The dynamics is the motion of a point particle 
traveling at unit speed and undergoing elastic reflections at the boundaries and is purely Hamiltonian. 
The associated two-dimensional collision map (the billiard map) preserves a smooth 
invariant measure $\mu_0$ with very strong ergodic properties: see the works by Sinai, Bunimovich and Chernov 
\cite{Sin70,BS,BSC,C99} on ergodicity,  mixing and the central limit theorem,  the proof  by Young \cite{Y98}  of  exponential decay of correlations, and many other statistical properties \cite{rey, melbourne nicol, melbourne nicol 2} as well as the recent proof by Baladi, Liverani and 
one of the authors \cite{BDL} for the  exponential decay of correlations for the billiard flow.  Of particular 
importance for this paper are the recent papers by two of the authors \cite{demers zhang,DZ2,DZ3}
who introduced Banach spaces suitable for a direct analysis of the dynamics by transfer operators 
which bypasses the construction of symbolic dynamics (Markov partitions  and Young towers). These
functional analytic tools will turn out to be crucial to prove the large deviation theorems needed in 
this paper.

Suitable perturbations of this model where the particle is submitted to external non-conservative 
forces in between or during collisions and to a suitable thermostatting mechanism have been put 
forward as simple, yet realistic, models in non-equilibrium statistical mechanics. 
With a constant  external electric field and an iso-energetic thermostat, this  kind of model was first
studied  by Chernov, Eyink, Lebowitz  and  Sinai \cite{CELSa,CELSb} who proved  
the existence of a unique SRB measure $\mu_+$ for the system:  for $\mu_0$ almost every initial 
condition
the system converges to an invariant measure  $\mu_+$ which is ergodic and mixing, 
and singular with respect to $\mu_0$.  In addition, they established  linear response formulas 
for this system. In subsequent papers, Chernov and one of the authors \cite{Ch01,Ch08,CZZ,Z11}  
generalized and strengthened these results to cover a large class of perturbations 
and our work  will rely on these results extensively.  In a more general context the use of 
thermostats and SRB measures as good models of non-equilibrium steady  states has 
been advocated, see e.g.~the book by Evans and Morriss \cite{EM} and  the papers
by Gallavotti and Cohen  \cite{GC1,GC2}  and Ruelle \cite{Ru99} (more on this  in Section \ref{sec:fluctuations}.)

%

One of the main results in this  paper is to establish a version of the so-called Gallavotti-Cohen 
fluctuation theorem \cite{GC1,GC2} for the entropy production  for the Lorentz gas driven out of 
equilibrium by  external forces.  The concept of entropy production  in non-equilibrium statistical 
mechanics,  in this context, was best formalized by Ruelle \cite{Ru96,Ru97,Ru99} (see also the earlier 
work by Andrey \cite{Andrey}) and we will discuss it in Section \ref{sec:fluctuations}.
The fluctuation theorem asserts that for time-reversible systems the time fluctuations (of large deviation type) 
of the entropy production have a universal symmetry:  the ratio of the probabilities of observing an 
average entropy production rate over a time interval of length $T$ equal to $a$ and equal to 
$-a$ is equal to $e^{aT}$.   The study of  the fluctuations of the entropy production for systems 
driven out of equilibrium originated in the numerical observation by  Evans, Cohen, and Morris 
\cite{ECM}  for a thermostatted  system driven by external shear.  The symmetry  of the transient 
fluctuations of  entropy production, that is when the system starts in the  equilibrium  (but not 
stationary) state ($\mu_0$ in our notation) was first noted by Evans and Searles \cite{ES1} 
(see Proposition \ref{ess} in Section \ref{sec:fluctuations}). On the other hand, using Markov partitions,  in 
\cite{GC1,GC2}  Cohen  and  Gallavotti established the fluctuation symmetry for time-reversible smooth uniformly 
hyperbolic systems starting in a  stationary  non-equilibrium state.  The relation between the transient 
and stationary fluctuation  theorem is discussed further in \cite{CG,ES2,luc entropy}. 
From a slightly different point of view,  Kurchan \cite{Ku}, Lebowitz and Spohn \cite{LS} proved the fluctuation theorem for general stochastic (Markovian and/or 
Gibbsian) dynamics and Maes \cite{Ma99} recast the fluctuation theorem as following from
the Gibbs 
property of an equilibrium state by considering the distribution of the time series of the process.  
Also in a related work,  
Jarzynski \cite{Ja}  established a very influential transient relation for the fluctuations  of work of 
a system driven by time-dependent forces.  These (and other) seminal works have given rise to a 
substantial amount of research in the past 20 years, and the fluctuation theorems and relations 
now stand as one of the pillars  in the modern theory of non-equilibrium statistical mechanics.  
There have been a number of recent reviews, among them \cite{Ma03,MN,ChGa,luc 
entropy}, to which we direct the reader for some of the recent developments in this subject.  Among 
these reviews,  Jaksic, Pillet,  and one of the authors \cite{luc entropy} present  a general formalism 
to understand the transient and stationary  fluctuation theorems,  and the relation between them, in the 
general framework of dynamical systems;  to some extent, we will follow the approach 
taken in that paper.

In this paper, we prove the steady state fluctuation relations for the 
periodic Lorentz gas with an external electric field and an iso-energetic thermostat 
\cite{CELSa,CELSb}  as well  as several classes of related models with different forcing 
mechanisms \cite{Ch01,Ch08,CZZ,Z11}.  While the models at hand are uniformly hyperbolic, 
the singularities of the billiard dynamics (due to grazing collisions)  preclude the use 
of Markov partitions to study the fluctuation properties of ergodic averages.  
Instead, we follow 
a direct approach using suitable transfer operators to express the cumulant generating function 
of ergodic averages. This approach to large deviations was used for non-uniformly 
hyperbolic dynamical systems in \cite{rey} using Young towers \cite{Y98}.  
Our approach consists in proving that the fluctuation properties of ergodic averages are the same 
for a  large class of initial distributions, which contains  both the stationary distribution $\mu_0$ of the 
Lorentz gas without external forces used to verify the transient fluctuation theorem, 
and the  invariant SRB measure for the perturbed Lorentz gas.  Since the symmetry of fluctuations when starting from 
$\mu_0$ is easy to establish (see Proposition \ref{ess}) a proof of the fluctuation theorem 
follows then immediately. The key new tool needed is the family of Banach spaces  
introduced by two of the  authors  \cite{demers zhang,DZ2,DZ3} to study the ergodic 
properties of billiards without  using the symbolic dynamics  tools used in earlier approaches (Markov 
partitions \cite{BS},  Markov sieves~\cite{BSC}, Young towers \cite{Y98, C99}). These Banach spaces are 
devised for the exact  purpose to be large enough to contain the SRB invariant measure, singular 
with respect to $\mu_0$ but smooth along unstable directions, yet small enough for the transfer 
operator to have a spectral gap. 
They also have the advantage of being stable under perturbations:  since all the relevant transfer
operators act on a single Banach space, we are able to show that important spectral quantities
vary smoothly as functions of certain system parameters, and from these properties we derive the
necessary control to prove the desired limit theorems.

This paper is organized as follows. In Section~\ref{sec:fluctuations} 
we give a brief overview of the ideas and
concepts of non-equilibrium statistical mechanics needed for the paper.  In Section~\ref{sec:results}
we introduce our model and state our main results. In Section~\ref{common}, following \cite{DZ2}, 
we discuss a general family of maps with singularities, to which our dynamical results apply. 
In Section~\ref{norms} we introduce the Banach spaces and transfer operators needed in our analysis. 
In Section~\ref{uniform} we prove the key analytical estimates needed to establish a spectral gap for the family 
of  transfer operators associated with the entropy production.  Finally, in Section~\ref{moment proof} we establish the 
analyticity and (strict) convexity of the logarithmic moment generating function, allowing us to 
conclude the proof of the fluctuation theorem. In Appendix~\ref{appendix} we provide the Lasota-Yorke 
estimates needed to establish a spectral gap for the relevant operators.


\subsection{Entropy production and fluctuation theorems}
\label{sec:fluctuations}

In this section, for the convenience of the reader, we provide a general (and somewhat informal) discussion, following \cite{luc entropy}, of the concepts of non-equilibrium steady states,  entropy production, 
and the fluctuation relations. 

The starting point is an invertible dynamical system $(M,T)$, i.e. a measurable space $M$ and an 
invertible measurable map  $T: M \to M$.  We also postulate the existence of a reference 
measure $\mu_0$  which, in general, is not  an invariant measure for $T$.  

In a physical context one may write $T=T_{\bf E}$ depending on some external non-equilibrium 
forces $\bf E$ with $T_{0}$ (for $\bE=0$) being the equilibrium dynamics without external forces. One may 
think   of  $\mu_0$ as the invariant measure for the dynamics $T_0$ without external forces; in
this context, 
$\mu_0$ is  the equilibrium steady state.  
If we think of $\mu_0$ 
as describing the initial state of the system, we then define $\mu_n$ as the state of the system 
at  time $n \in \mathbb Z$, i.e. we have 
\begin{equation}
\label{eq:evolve}
\mu_n ( f)  = \mu_0( f \circ T^n) \,,
\end{equation}
for any bounded measurable $f$.  

We introduce next the concept of a non-equilibrium steady state following Ruelle \cite{Ru99}. 
\begin{defn} 
A probability measure $\mu_+$ is called a non-equilibrium steady state for the  dynamical system
$(M,T)$ with reference measure $\mu_0$ if: 
\begin{enumerate}
\item the measure $\mu_+$ is an ergodic invariant measure for $T$;
\item for $\mu_0$-almost every initial condition $x \in M$ the empirical measure $\frac{1}{n} \sum_{k=0}^{n-1} \delta_{T^k(x)}$ converges weakly to $\mu_+$ as $n \to \infty$; 
\item the measure $\mu_+$ is singular with respect to $\mu_0$. 
\end{enumerate}
\end{defn} 
Item (2) in the definition selects one invariant measure $\mu_+$ among the usually many invariant 
measures of the dynamical system $(M,T)$ and it is essentially equivalent to the SRB property in the 
theory of  hyperbolic  dynamical systems if $M$ is a smooth manifold and $\mu_0$ is Lebesgue measure.
Measures satisfying (2) are also often called ``physical measures" as they describe the statistics of 
``most" initial conditions.  Item (3) in the definition ensures that   the invariant  measure is truly a 
``non-equilibrium" steady  state in the sense of statistical mechanics,  
while if $\mu_+$ were equivalent to $
\mu_0$ it should rather be called an equilibrium steady state. Finally in a physical context where $T=T_{\textbf{E}}$ depends on external forces, the non-equilibrium steady state $\mu_+$ depends 
on $\textbf{E}$ and we will use the notation $\mu_{\textbf{E}}$ in that case.

Next we turn to the concept of the entropy production observable $s: M \to \mathbb{R}$ 
which plays a central role in non-equilibrium statistical mechanics.  We make the (rather weak) regularity assumption that  $\mu_n$ and $\mu_0$ are mutually absolutely continuous and denote by $l_n$ the logarithm of the Radon-Nykodym derivative,
\[
l_n = \log \frac{d \mu_n}{d\mu_0} \,.
\]
Since 
$
\mu_{n+m}(f) = \mu_m(f\circ T^n) = \mu_0( e^{l_m} f\circ T^n) = \mu_n( e^{l_m\circ T^{-n}} f) = \mu_0( e^{l_n} e^{l_m\circ T^{-n}} f)$,
we have the chain rule, $ l_{n+m} = l_n + l_m \circ T^{-n}$,
and in particular, $l_{-1} =- l_1 \circ T$. Therefore, we have 
\[
l_n = \sum_{k=0}^{n-1}  l_1 \circ T^{-k} \,. 
\]
For two probability measures $\mu$ and $\nu$ on $M$, let us denote by $R(\mu | \nu)$ the 
relative entropy of $\mu$ with respect to $\nu$ (also known as the Kullback-Leibler divergence) 
which is defined by 
\[
R(\mu | \nu) \,=\, \left\{ 
\begin{array}{cl}   
\int \log \frac{d\mu}{d\nu}   d\mu & \textrm{if } \mu \ll \nu \\
+\infty     &   \textrm{otherwise} 
\end{array}
\right. \,.
\]
We have then
\[
R(\mu_n | \mu_0) = \mu_n ( l_n) = \mu_n \left(\sum_{k=0}^{n-1}  l_1 \circ T^{-k}\right) \,=\, \mu_0 \left( \sum_{k=1}^n l_1 \circ T^k \right) ,
\]
using \eqref{eq:evolve}.
This leads to the following definition.
\begin{defn}
The entropy production observable for the  dynamical  $(M,T)$ with reference measure $\mu_0$ is  
is given by 
\[
s = l_1 \circ T \,.
\]
\end{defn}
If we assume the existence of a non-equilibrium steady state and if the entropy production observable 
$s$ is regular enough we have  
\[
\lim_{n \to \infty} \frac{1}{n} R(\mu_n| \mu_0) =  \lim_{n \to \infty} \mu_0\left(\frac{1}{n} \sum_{k=0}^{n-1} s \circ T^k\right) =\mu_+(s)  \ge 0 \,,
\]
since the relative entropy is non-negative.  This general fact is known as the non-negativity of 
the entropy production rate in non-equilibrium steady states.  It is shown in \cite[Section 5]{luc entropy}  that, under quite general 
conditions, we have 
\[
\mu_+(s)>0 \textrm{ if and only if } \mu_+ \textrm{ is singular with respect to } \mu_0\,.
\]
We expect in any case that, for  a bona fide non-equilibrium steady state,  
we have positivity of entropy  production, i.e., $\mu_+(s)>0$, a fact   
which usually requires some non-trivial analysis.  
We prove this result in the context of the Lorentz gas under external forces as part of 
Theorem~\ref{thm:moment}.

An important example in the context of this paper is when the state space $M$ is a smooth manifold, 
$\mu_0$ is a measure with a smooth density with respect to Lebesgue measure on $M$, and $T$ 
is  a (piecewise) smooth transformation. In this case the change of variable formula gives 
\[
e^{l_n} \,=\, \frac{1}{J_{\mu_0}T^n \circ T^{-n}}\,,  
\]
where $J_{\mu_0}T$ is the Jacobian of the map $T$ with respect to $\mu_0$ and  
therefore
\[
s = - \log J_{\mu_0}T \, ,
\]
which can be interpreted as describing a phase space contraction rate.  We refer to \cite{luc entropy} for 
various other examples.

The fluctuation theorem asserts that the fluctuations of the ergodic averages of the  entropy production 
have a universal symmetry under the condition that the system is invariant under-time reversal.
\begin{defn} The  dynamical system $(M,T)$ with reference measure $\mu_0$ is  time-reversal invariant 
if there exists an involution $i: M \to M$  (that is, $i \circ i$ is the identity) such that, 
\begin{enumerate}
\item $\mu_0$ is invariant under $i$, i.e., $\mu_0( f \circ i) = \mu_0(f)$;   
\item $i \circ T \circ i = T^{-1}$. 
\end{enumerate}
\end{defn}
Using the time reversal property, we have for any bounded measurable $f$, 
\[
\mu_0(e^{l_{-n}} f)   = \mu_{0} (f \circ T^{-n} ) = \mu_{0} (f \circ T^{-n} \circ i) = \mu_{0} (f \circ  i \circ T^{n}) = \mu_0(e^{l_n} f \circ i) =  \mu_0(e^{l_n \circ i} f ) ,
\]
and hence
\begin{equation}
\label{eq:ell n}
l_{-n} = l_n \circ i\,.
\end{equation}
Using this it is straightforward to derive the so-called transient fluctuation theorem \cite{ES1,luc entropy} (also called the Evans-Searles fluctuation theorem). We give a proof  here for the convenience of the reader.

\begin{proposition}[\bf Transient fluctuation theorem]\label{ess}  Suppose the dynamical system
$(M,T)$ with reference measure $\mu_0$ is time-reversal invariant  and $s$ is the  entropy production observable.
Then we have the symmetry
\[
\mu_0\left( e^{- a \sum_{k=0}^{n-1} s \circ T^k} \right) \,=\, \mu_0\left( e^{- (1-a) \sum_{k=0}^{n-1} s \circ T^k} \right)  ,
\]
for any $a \in \mathbb{R}$ for which both integrals are finite.
\end{proposition}
\begin{proof}
First we use that by the chain rule, 
\[
l_{-n} = - l_n \circ T^n= - \sum_{k=1}^n l_1\circ T^k = - \sum_{k=0}^{n-1} s \circ T^k .
\]
Thus without the assumption of time reversal, we have by \eqref{eq:evolve}  
\begin{equation}\label{eq:es2}
\mu_0(e^{- a \sum_{k=0}^{n-1} s \circ T^k}) 
= \mu_0( e^{- a l_{n}\circ T^n}) = \mu_n(e^{- a l_{n}}) = \mu_0(e^{(1- a) l_{n}})\,.
\end{equation}
On the other hand time reversal  implies by \eqref{eq:ell n} that,
\begin{equation}\label{eq:es3}
\mu_0(e^{(1-a) l_n}) = \mu_0(e^{(1-a) l_{-n}\circ i}) = \mu_0(e^{(1-a) l_{-n}})
= \mu_0 (e^{-(1-a) \sum_{k=0}^{n-1} s \circ T^k}) \, .
\end{equation}
Combining \eqref{eq:es2} and \eqref{eq:es3} gives the desired symmetry. 
\end{proof}

The transient fluctuation theorem has the following interpretation (Proposition 3.3 of \cite{luc entropy}):
if $P_n(z)$ denotes the probability distribution of $\sum_{k=0}^{n-1} s \circ T^k$ with initial distribution $\mu_0$ and $\tau(z) = -z$ then we have 
\[
\frac{dP_n }{dP_n \circ \tau} = e^{nz},
\]
which gives a universal ratio for the probabilities to observe an average entropy production rate 
equal to $+z$ or $-z$.   

By contrast the Gallavotti-Cohen (steady state) fluctuation relation deals with the fluctuation starting in 
the non-equilibrium steady state $\mu_+$.   To state it we define, for any probability measure $\nu$,  
the logarithmic moment generating function
\[
e_\nu(a) = \lim_{n \to \infty} \frac{1}{n} \log \nu \left( e^{- a \sum_{k=0}^{n-1}
 s \circ T^k} \right) ,
\]
provided the limit exists.

\medskip 
\noindent
{\bf Steady state fluctuation relation}.   {\em The  dynamical system $(M,T)$ with reference measure 
$\mu_0$ and non-equilibrium steady state $\mu_+$ satisfies the steady state fluctuation relation 
if for some $a_0 > 0$ and all $a \in [-a_0, 1+a_0]$: 
\begin{enumerate}
\item the limit defining the logarithmic moment generating function exists,
\[
e_{\mu_+}(a) \,=\, \lim_{n \to \infty} \frac{1}{n} \log \mu_+ \left( e^{- a \sum_{k=0}^{n-1}
 s \circ T^k} \right) ;
\] 
\item the moment generating function has the following symmetry, 
\[
e_{\mu_+}(a) = e_{\mu_+}(1-a) .
\]
\end{enumerate}
}

The transient and steady state fluctuation relations look similar, yet are distinct statements. 
In particular, the transient fluctuation theorem is a finite time statement, valid even in the absence of a 
steady state.  Even if we assume that the limit $e_{\mu_0}(a)$ exists (a nontrivial statement), 
one cannot expect, in  general, that  $e_{\mu_0}(a) = e_{\mu_+}(a)$ even if $\mu_+$ is a steady state (with reference measure $\mu_0$) (see e.g. \cite{CG} for a counterexample).   
There certainly are examples where these two functions coincide, e.g. for Anosov diffeormorphisms 
(see e.g. \cite{luc entropy}) and indeed one of the main contributions of this paper is to 
prove that for 
billiards under small external forces the limits $e_{\mu_0}(a)$ and $e_{\mu_+}(a)$ exist and coincide
for a non-perturbative range of values of the parameter $a$.

To conclude we briefly discuss the large deviation interpretation of the symmetries.  From the theory of 
large  deviations, it is well known that 
if $e_\nu(a)$ is $\C^1$ on an interval $a \in [-a_0, 1+a_0]$,  then by the Gartner-Ellis theorem (see \cite{DZe}) we have a large 
deviation principle for the ergodic averages $\frac{1}{n} \sum_{k=0}^{n-1} s \circ T^k$, 
with initial condition distributed according to $\nu$, i.e., 
$$
\lim_{\delta \to 0}\lim_{n\to\infty} \frac{1}{n} \log \nu\left(x: \frac{1}{n}\sum_{k=0}^{n-1} s \circ T^k \in [z-\delta, z+\delta]\right)=-I(z),$$
for any $z\in [e'_\nu(-a_0), e'_\nu(1+a_0)]$, where $I:\mathbb{R}\to[0,\infty]$ 
is the rate function given by the Legendre transform 
\[
I(z) \,=\,  \sup_{- a_0 \le a \le 1+a_0} \{ a z - e_\nu(a) \}\,.
\]
The symmetry  $e_{\nu}(a) =e_{\nu} (1-a)$ implies that rate function $I(z)$ has the symmetry 
\begin{equation}
\label{eq:sym I}
\begin{split}
I(z)  & = \sup_{- a_0 \le a \le 1+a_0} \{ az - e_\nu(a) \} = \sup_{-a_0 \le a \le 1+a_0} \{ az - e_\nu(1-a)\} \\
& = \sup_{-a_0 \le b \le 1+a_0} \{ (1-b) z - e_\nu(b)\} 
\,=\, I(-z) -z \,.
\end{split} 
\end{equation}
The symmetry of the rate function $I(z)-I(-z)=-z$ implies that the ratio of  probabilities 
to observe an entropy production rate equal to $z$ and equal to $-z$ over a time interval of length $n$ is asymptotically equal to $e^{nz}$.

One can also show that the fluctuation relation does imply the Kubo formula for the linear response of 
currents, but we shall not discuss this further here (see e.g. \cite{LS,Ma99,Ma03,luc entropy}). 

%
%
%
%
%
%
%
%


\section{Description of Model and Main Results}
\label{sec:results}

Let $d>1$, we define a periodic Lorentz gas by placing finitely many
closed, convex regions (scatterers) $\Gamma_i$, $i=1, \ldots d$,
on a Torus $\mathbb{T}^2$, which are pairwise disjoint and have $\C^3$ boundaries with strictly positive
curvature.
The classical billiard flow on the table
$\mathbb{T}^2 \setminus \cup_i \{ \mbox{interior } \Gamma_i \}$
is defined by the motion of a particle traveling at unit speed and undergoing elastic
collisions at the
boundaries.  In this paper we will also consider the motion of particles
subject to external forces,
as well as certain types of collisions which do not obey the
usual law of reflection.

The discrete-time billiard map $T$ associated with the flow is the Poincar\'e map
corresponding to collisions with the scatterers.  At each collision, we record the position 
according to an arclength parameter $r$ (oriented clockwise on the boundary of each scatterer)
and the angle $\vf$ made by the outgoing (post-collision) velocity with the unit normal to the boundary 
at the point of collision.  
The phase space of the map is thus $M = \cup_{i=1}^d I_i \times [-\pi/2, \pi/2]$,
where each $I_i$ is an interval with endpoints identified and with length equal to the arclength
of $\partial \Gamma_i$.

For any $x = (r,\vf) \in M$, define $\tau(x)$ to be the free path of the first collision
of the
trajectory starting at $x$ under the billiard flow. The billiard map is
defined wherever
$\tau(x) < \infty$. We say that the billiard has finite horizon if there
is an upper bound
on the function $\tau$. Otherwise, we say the billiard has infinite
horizon.
Notice that the function $\tau$ depends on the (possibly curved) trajectories of 
particles in $\mathbb{T}^2$, 
while $M$ is independent of the trajectories; thus we may study many classes of perturbations
of a billiard flow while fixing $M$.

We will denote by $d\mu_0 = c_0 \cos \vf dr d\vf$ the smooth invariant probability measure
which is preserved by the unperturbed billiard map, where $c_0$ is the normalizing constant.


\subsection{Assumptions}
\label{model}

In this subsection we  first state the assumptions on the model, following \cite{CZZ} 
(which in turn combines the assumptions in \cite{CZ09, Z11, DZ2}). 

Let $\bq = (x,y)$ be the position of a particle in the billiard table
$Q := \mathbb{T}^2 \setminus (\cup_i \Gamma_i)$ and
$\p$ be the velocity vector.  
We may define a perturbed billiard flow on $Q$ as follows.
Between collisions, the position and velocity obey the following differential equation,
\begin{equation}\label{flowf}
    \frac{d \bq}{dt} =\p(t) , \qquad
    \frac{d \p}{dt} = \mathbf{F}(\bq, \p) ,
\end{equation}
where
$\mathbf{F}: \mathbb{T}^2 \times \mathbb{R}^2 \to \mathbb R^2$ 
is a $C^2$ stationary external force.
At collisions, the trajectory experiences possibly nonelastic reflections with slipping along
the boundary,
\begin{equation}
\label{reflectiong}
(\bq^+(t_i), \p^+(t_i)) = (\bq^-(t_i), \cR  \p^-(t_i)) +\mathbf G(\bq^-(t_i), \p^-(t_i)) ,
\end{equation}
 where $\cR \p^-(t_i)= \p^-(t_i)+2(n(\bq^-)\cdot \p^{-})n(\bq^-)$ is the usual reflection operator,
 $n(\bq)$ is the unit normal vector to the billiard wall $\partial Q$ at
 $\bq$ pointing inside the table $Q$, and $\bq^-(t_i), \p^-(t_i)$, $\bq^+(t_i)$ and
 $\p^+(t_i)$ refer to the incoming and
 outgoing position and velocity vectors, respectively.
 $\mathbf G$ is an external force acting on the
 incoming trajectories.
 We allow $\bf G$ to change both the position and the velocity of the particle at the moment
 of collision.  The change in velocity can be thought of as a kick or twist while a change in
 position can model a slip along the boundary at collision, or even reflection by a soft billiard
 potential \cite{balint toth}.

In \cite{Ch01, Ch08}, Chernov considered billiards under small external forces $\mathbf{F}$
with $\mathbf G=0$, and $\mathbf{F}$ to be stationary. In \cite{Z11} a twist force was considered assuming $\mathbf{F}=0$ and $\bG$
depending on and affecting only the velocity, not the position.
Here we follow \cite{DZ2, CZZ} and consider a combination of these two cases for systems under more general forces $\mathbf{F}$ and $\mathbf{G}$.

Let
$\bE=(\bF,\bG)$,
where $\bF$ and $\bG$ are the two external forces
during the flight and at collisions, respectively.
Let $\Phi_{\bE}^t$ be the induced billiard flow on $Q \times \mathbb{R}^2$
and denote by $T_{\bE} = T_{\bF, \bG}$ the corresponding billiard map.

\bigskip
\noindent\textbf{(A1)} (\textbf{Invariant space}) \emph{The
perturbed flow $\Phi_{\bE}^t$ preserves a smooth function $\cE(\bq,
\bp)$, such that the level surface $\cM:=\{\cE(\bq,\bp)=c\}$ is a
compact 3-D manifold, for some $c>0$.
Moreover, $\| \bp\|>0$ on $\cM$, and for each $\bq\in Q$ and $ \bp\in
S^1$, the ray $\{(\bq, t \bp), t>0\}$ intersects the manifold $\cM$ in
exactly one point.} 

\bigskip
Under assumption ({\textbf{A1}}), the system has an additional integral of
motion and we will consider the restricted system on a compact
phase space, $\cM \subset Q \times \mathbb{R}^2$. 
For example, if we add a Gaussian thermostat (a heat
bath) to the system such that the billiard moves at constant speed
(constant temperature if there are a large number of particles),
then $\cM:=\{\|\bp\|=c\}$ is an invariant compact level set.
More generally, the speed $p=\| \bp\|$ of the
billiard along any typical trajectory on $\cM$ at time $t$ satisfies
$$0<p_{\min}\leq p(t) \leq p_{\max}<\infty,$$
for some constants $p_{\min}\leq p_{\max}$. 
In addition, $\cM$ admits a
global coordinate system $\{(x,y,\theta):(x,y)\in
Q,0\le\theta<2\pi\}$, where $\theta$ is the angle between $\bp$ and
the positive $x$-axis. Thus the speed $p=\|\bp\|$ on $\cM$
can be represented as a function $p=p(x,y,\theta)$ and the
velocity $\bp$ at $\bq$ can be expressed as $\bp=p\bv$, where
$\bv=(\cos\theta, \sin\theta)$ is the unit vector in the
direction of $\bp$.  We can then rewrite eq.~\eqref{flowf} for the dynamics between
collisions  as 
\beq
\label{vpF} 
\dot \bq = \bp, \quad
\dot p \bv+p\dot\bv=\bF. 
\eeq 
Multiplying both sides of
(\ref{vpF}) by $\bv$ using the dot product and cross product respectively, we
obtain 
\beq
\label{pvF} \dot p= \bv\cdot \bF,\,\,\,\,\,\text{ and
}\,\,\,\,\,\,p\bv\times \dot\bv=\bv\times \bF. 
\eeq 
Therefore, using the notation $\bF = (F_1, F_2)$,  the
equations in (\ref{flowf}) have the following coordinate
representations at any $(x,y,\theta)\in\cM$, 
\beq
\label{flow}
\left\{
  \begin{array}{ll}
    \dot x=p\cos\theta, \\
    \dot y=p\sin\theta, \\
  \dot \theta=(-F_1\sin\theta+F_2\cos\theta)/p.
  \end{array}
\right.
\eeq
Next, consider a trajectory $\tilde\gamma\subset\cM$ of the flow passing
through the point $(x,y,\theta)\in\cM$,
which projects down to a smooth curve $\gamma\subset Q$.
We denote by $\kappa = \kappa(x,y,\theta)$ the (signed) geometric curvature of $\gamma$ at
$(x,y)\in Q$.
It follows that
\beq
\label{eq:formh}
\kappa(x,y,\theta)=\pm\frac{\|\dot\bq\times\ddot\bq\|}{\|\dot\bq\|^3}=\pm\frac{\|\bv\times
\bF\|}{p^2}
=\frac{-F_1\sin\theta+F_2\cos\theta}{p^2},
\eeq
where the sign should be chosen accordingly. Combining this with
(\ref{flow}), we have
\beq\label{ptheah}
\dot\theta= p \kappa.
\eeq
Note that the angle $\theta=\theta(t)$ is discontinuous at reflection
times: it jumps from $\theta^-$ to $\theta^+$.
In the case of elastic collisions, the quantities $x$, $y$ and $p$
remain unchanged.
By contrast, under the twisting force $\bG$, all quantities may change at collisions.\\

For any point $(x, y,\theta) \in \cM$, let
$\tau(x, y,\theta)$ be the time for the trajectory starting from
$(x, y, \theta)$ to make its next non-tangential collision at $\partial Q$. \\

\noindent(\textbf{A2}) (\textbf{Finite horizon})
\emph{There exist $\tau_{\max}>\tau_{\min}>0$ such that free paths between
successive non-tangential reflections are uniformly bounded:
$\tau_{\min} \le \tau(x,y,\theta ) \le \tau_{\max}$, for all $(x,y,\theta)
\in \cM$ with $(x,y) \in\pd Q$.}
 \\

\noindent(\textbf{A3}) (\textbf{Smallness of the external forces}).
\emph{There exists
$\eps>0$ small enough such that the forces
$\bE=(\bF,\bG)$ satisfy
$$\|{\bF}\|_{C^1}<\eps, \|\bG-\mathrm{Id}_M\|_{C^1}<\eps.$$
Moreover, there exist constants $\alpha_0 > 1/3$ and $C_{\bE} > 0$ such that 
$\| \bF \|_{C^{1+\alpha_0}}, \| \bG \|_{C^{1+\alpha_0}} \le C_{\bE}$.}

\begin{remark}
Note that (\textbf{A2}) also puts some implicit constraints on the smallness of
forces. In fact, the existence of $\tau_{\min}$ not only prevents touching
scatterers, but also implies the trajectory cannot be bent too much such that the
particle falls back to the same scatterer immediately.
\end{remark}

Let $\cI:\cM\to\cM$ be the 
involution
defined by $\cI(x,y,\theta)=(x,y,\pi+\theta)$. 
For a general flow
$\Phi^t:\cM\to\cM$,
the reversed flow of $\Phi^t$ is defined by $\Phi^t_{-}=\cI\circ
\Phi^{-t}\circ \cI$.
The flow $\Phi^t$ is said to be {\it time-reversible}, if
$\Phi^t_-=\Phi^t$.
It is well known that the unforced billiard flow is time-reversible. \\

\noindent(\textbf{A4}) (\textbf{Time-reversibility})
\emph{Both forces $\bF$ and $\bG$ are stationary,
and the forced billiard flow $\Phi_{\bE}^t$ is time-reversible.
Moreover, we assume that $\bG$ preserves tangential collisions:
$\bG(r,\pm \frac{\pi}{2})=(r,\pm \frac{\pi}{2})$. } \\

Note that due to ({\bf A4}), the singularity set of $T^{-1}_{\bF, \bG}$ is the same as that
of the untwisted map $T^{-1}_{\bF, \mathbf{0}}$. It also implies that the billiard map 
$T_{\bf E}$ is time-reversible.

Fix $\eps_0 >0$, $\tau_\ast \in (0,1)$, and $C_0 > 0$.  For the fixed billiard table $Q$, let 
$\cF(\eps_0, \tau_\ast, C_0)$ denote
the collection of all forced billiard maps defined by the dynamics \eqref{flowf} and 
\eqref{reflectiong} under the external forces $\mathbf{E}=(\bF,\bG)$ and satisfying assumptions
(\textbf{A1})--(\textbf{A4}), such that
$\tau_\ast\le\tau_{\min}\le\tau_{\max} \le \tau_\ast^{-1}$, $C_{\bE} \le C_0$,
and $\ve \le \ve_0$ in ({\bf A3}). 

In Section~\ref{class of maps} we define a class of maps satisfying uniform properties 
regarding hyperbolicity and singularities, {\bf (H1)}--{\bf (H5)}.
The following lemma from \cite{DZ2} is crucial in that respect.

\begin{lemma}$($\cite[Theorem 2.10]{DZ2}$)$
\label{lem:verify H}
Fix $\tau_\ast \in (0,1)$.  There exist $\ve_0, C_0 > 0$ such that 
the family of maps $\cF(\eps_0, \tau_\ast, C_0)$ satisfy
(\textbf{H1})--(\textbf{H5}) with uniform constants.
\end{lemma}


\subsection{Transfer operators}
\label{sec:transfer}

In this section, we fix a class of maps $\cF$ with uniform properties {\bf (H1)}--{\bf (H5)}
as defined Section~\ref{class of maps}.  Later, we will specialize to a particular
family $\cF = \cF(\ve, \tau_\ast, C_0)$ satisfying {\bf (A1)}--{\bf (A4)} above.

Let $\widehat{\cW}^s$ be the set of stable curves invariant under maps in
$\cF$ according to {\bf (H2)}, and let $\cW^s \subset \widehat{\cW^s}$ denote those stable curves
having length less than $\delta_0$, where $\delta_0$ is from \eqref{eq:one step contract}.  
For any $T \in \F$,
we define scales of spaces
using the set of stable curves $\cW^s$ on which the
{\em transfer operator} $\Lp_T$ associated with $T$ will act.
Define $T^{-n}\cW^s$
to be the set of homogeneous stable curves $W$ such that $T^n$ is smooth
on $W$ and
$T^iW \in \cW^s$ for $0 \leq i \le n$.   It follows from {\bf (H2)}  that
$T^{-n}\cW^s \subset \cW^s$.

For $W \in T^{-n}\cW^s$, a complex-valued test function $\psi: M \to
\mathbb{C}$, and $0<\alpha \le 1$ define $H^\alpha_W(\psi)$ to be
the H\"older constant of $\psi$ on $W$ with exponent $\alpha$ measured in the
Euclidean metric.
Define $H^\alpha_n(\psi) = \sup_{W \in T^{-n}\cW^s} H^\alpha_W(\psi)$
and let $\tilde{\C}^\alpha(T^{-n}\cW^s) = \{ \psi : M \to \mathbb{C} \mid
H^\alpha_n(\psi) < \infty \}$,
denote the set of complex-valued functions which are H\"older continuous
on elements of
$T^{-n}\cW^s$.
The set $\tilde{\C}^\alpha(T^{-n}\W^s)$ equipped with the norm
$| \psi |_{\C^\alpha(T^{-n}\W^s)} = |\psi|_\infty + H^\alpha_n(\psi)$ is a Banach
space.
Similarly, we define $\tilde \C^\alpha(\widehat \W^u)$ to be the set of functions
which are H\"older continuous
with exponent $\alpha$ on unstable curves $\widehat \W^u$.

It follows from the uniform hyperbolicity of $T$ (see {\bf (H1)}) that if $\psi \in
\tilde{\C^\alpha}(T^{-(n-1)}\cW^s)$, then
$\psi \circ T \in \tilde{\C}^\alpha(T^{-n}\cW^s)$.  Thus
if $h \in(\tilde \C^\alpha(T^{-n}\cW^s))'$, is an element of the dual of $\tilde
\C^\alpha(T^{-n}\cW^s)$,
then
$\Lp_T :(\tilde \C^\alpha(T^{-n}\cW^s))'\to (\tilde \C^\alpha(T^{-(n-1)}\cW^s))'$ acts
on $h$ by
\[
 \Lp_T h(\psi) := h(\psi \circ
T) \quad \forall \psi \in \tilde \C^\alpha(T^{-(n-1)}\cW^s).
\]
Recall that $d\mu_0 = c \cos \vf dr d\vf$ denotes the smooth invariant
measure for the billiard map corresponding to the unperturbed periodic
Lorentz gas.  
If $h \in L^1(M,\mu_0)$, then $h$ is canonically identified with a signed
measure
absolutely continuous with respect to $\mu_0$, which we shall also call $h$,
i.e.,
$
h(\psi) = \int_M \psi h \, d\mu_0.
$
With the above
identification, we write $L^1(M,\mu_0) \subset (\tilde \C^\alpha(T^{-n}\cW^s))'$
for each $n \in \N$.
Then restricted to $L^1(M,\mu_0)$, $\Lp_T$ acts according to the
familiar expression
\[
\Lp_T^n h = \frac{h \circ T^{-n}}{J_{\mu_0}T^n \circ T^{-n}}  \; \; \;
\mbox{for any $n \geq 0$ and $h \in L^1(M,\mu_0)$,}
\]
where $J_{\mu_0}T$ is the Jacobian of $T$ with respect to $\mu_0$.


In Section~\ref{norms}, we define Banach spaces of distributions
$( \B, \| \cdot \|_\B )$ and $(\B_w, | \cdot |_w )$, preserved under the
action
of $\Lp_T$,  such that the unit ball of $\B$ is compactly embedded
in $\B_w$.  It follows from \cite[Corollary 2.4]{DZ2} that for $\ve$ sufficiently
small, $\Lp_T$ has a spectral gap on $\B$.

To study large deviations we will need a suitable weighted transfer operator.  
In order to have a well defined operator on $\B$ we will assume that is $g:M\to \R$ is (piecewise) 
H\"older continous on the connected components  of $M \setminus \mathcal{S}_1^T$ where 
$\mathcal{S}_1^T$ is the set of discontinuities of $T$ (see Sections \ref{common} \and \ref{recall property} for details).  
Under these assumptions it is shown in Lemma \ref{lem:L well defined}  that  we can define the 
weighted  transfer operator  $\Lp_{T,g}$ associated with $T$ and $g$ on $\B$ and $\B_w$ by
\begin{equation}
\label{eq:gen trans}
\Lp_{T,g}h(\psi):= \Lp_T (he^g) (\psi) = h( e^g \cdot \psi \circ
T), \;\; \; \mbox{for $h \in \B_w$ and suitable test functions } \psi.
\end{equation}
The family of transfer operators $\Lp_{T, ag}$ parametrized by $a\in \mathbb{R}$ 
occurs naturally in studying the large deviations of  Birkhoff sums 
$S_ng=g+\cdots +g\circ T^{n-1}$:  since we have 
\[
\Lp_{T, ag}^n h(\psi) = h(e^{a S_ng} \psi \circ T^n) \,,
\]
the logarithmic moment generating function of $S_n g$  with initial distribution  
$\nu \in \B$ is then given by 
\[
\log \nu(e^{a S_n g})= \log \Lp_{T, a g}^n \nu(1)\,.
\]
Suitable spectral gap conditions on  $\Lp_{T, a g}$ imply 
that the limit 
\[
e_\nu(a)=\lim_{n\to\infty}\frac{1}{n}\log \nu(e^{a S_n g})  =  \lim_{n\to\infty}\frac{1}{n} \log \Lp_{T, a g}^n \nu(1)\
\]
exists and is smooth and then large deviation estimates follow from the G\"artner-Ellis theorem
\cite{DZe}.  In this paper we shall be interested in particular in the choices $\nu=\mu_0$,
the SRB measure for the unperturbed Lorentz gas,  and $\nu=\mu_\mathbf{E}$, the SRB measure 
for the perturbed Lorentz gas $T_\mathbf{E}$, both measures belonging to $\B$. 

%
%
%
%
%
%

\subsection{Statement of Results}
\label{results}

In \cite{DZ2}, local large deviation estimates for (piecewise) smooth observables $g$ 
results were obtained for small $\ve$ (small forces) and small $a$ 
(deviations very close to the mean of $g$); these were essentially 
perturbative results in $a$ and $\ve$.   By contrast here  we  concentrate on the observable 
$s =- \log J_{\mu_0}T_{\bf E}$, which is   the entropy production observable defined in  
Section \ref{sec:fluctuations}.  For the fluctuation symmetry to 
make sense we will need  the moment generating function to be well-defined for $a$ in a 
neighborhood of $[0,1]$. To this end, we will fix $a_0 >0$ and consider the interval $a \in [-a_0, 
1+a_0]$.  We study the dependence of  the spectral gap of $\Lp_{T, -as}$ as a function of the two 
parameters, $\ve$ and $a$. Since $s$ is fixed, in what follows we will use the more concise 
notation,   $\Lp_{T, a} = \Lp_{T, -as}$.  Note also that in the absence of external forces, $\mu_0$
is an invariant measure and $J_{\mu_0}T_0=1$.   More generally, for $T_{\bf E} = T_{(\bF, \bG)}$, 
we show in  Lemma \ref{TFJ} that 
\[
 J_{\mu_0}T_{\bf E} = 1 + \ve H
 \]
where $H$ is bounded uniformly in $\ve$, a key fact in our analysis.

The following spectral result is key to proving the existence and smoothness
of the limiting logarithmic moment generating function.

\begin{theorem}[Spectral gap]
\label{thm:uniform}
Choose $a_0 > 0$ and fix the parameters $C_0, \tau_*$ from Section~\ref{model}.  
There exists $\ve_0 > 0$ such that for any $T \in \F := \F(\ve_0, \tau_*, C_0)$,
 the operator
 $\Lp_{T,a}$ is well defined
as a bounded linear operator on $\B$ for all $a \in [-a_0, 1+a_0]$.  
In addition, there exists $C>0$,
such that for any $T \in \F$ and $n \geq 0$,
\begin{equation}
\begin{split}
\label{eq:uniform LY}
| \Lp_{T,a}^n h|_w & \leq C (1+\sign(a-1) C_H \ve)^{n(a-1)} |h|_w \qquad \mbox{for all
$h \in \B_w$}, \\
\| \Lp_{T,a}^n h \|_\B & \le C \sigma^n (1+\sign(a-1) C_H \ve)^{n(a-1)} \| h \|_\B + C \eta^n
|h|_w \qquad \mbox{for all $h \in \B$},
\end{split}
\end{equation}
where $C_H > 0$ is from Lemma~\ref{TFJ} and $\sigma\in(0,1)$ is from \eqref{eq:LY}. 
Moreover, for
each $T \in \F$,
\begin{itemize}
\item[(i)] $\Lp_{T,a}$ is quasi-compact as an operator on $\B$:  The spectral radius $\rho(\Lp_{T,a})$ lies in $[(1- \sign(a-1) C_H\eps_0)^{a-1}, (1+ \sign(a-1)C_H\eps_0)^{a-1}]$, while the 
essential
spectral radius $\rho_{\text{ess}}(\Lp_{T,a})$ is at most $\sigma (1+ \sign(a-1)C_H\eps_0)^{a-1} < (1- \sign(a-1) C_H\eps_0)^{a-1}$.
\item[(ii)] There exists $\ve_1 \le \ve_0$ such that for all $T \in \F(\ve_1, \tau_*, C_0)$ and all
$a \in [-a_0, 1+a_0]$,
$\Lp_{T, a}$ has a spectral gap:  there exists exactly one simple real eigenvalue
$\lambda_a=\rho(\Lp_{T,a})$; the corresponding eigenfunction $h_{a}$ is a
positive Borel measure.
\end{itemize}
\end{theorem}

For $T_{\bE} \in \F(\ve_1, \tau_*, C_0)$, we discuss next the existence and properties of the logarithmic moment generating function 
for the entropy production observable $s = - \log J_{\mu_0}T_{\bf E}$ with respect to 
the non-equilibrium 
steady state $\mu_{\bf E}$,
\begin{equation}
\label{eq:moment}
e_{\bE}(a)=\lim_{n\to\infty}\frac{1}{n} \log \mu_{\bE}\left((J_{\mu_0}T_{\bE}^n)^a\right).
\end{equation}
We denote by $\sigma_{\bE}^2$ the diffusion constant for the sequence 
$\{\log J_{\mu_0}T_{\bf E} \circ T_{\bE}^n\}_{n \ge 0}$ distributed according to $\mu_{\bE}$,
and by $\sigma^2_H$ the diffusion constant for the sequence $\{H\circ T_0^n\}_{n \ge 0}$
distributed according to $\mu_0$.
Our main results are summarized in the following theorem
\begin{theorem}[Logarithmic moment generating function and fluctuation relation]
\label{thm:moment}  
Under the assumptions  of Theorem \ref{thm:uniform},  we have the following.

\begin{enumerate}
\item The map $T_E$ has a unique SRB measure (non-equilibrium steady state) $\mu_{\bf E}$. 
\item The logarithmic moment generating function $e_{\bE}(a)$ for the entropy production 
exists and is analytic in the disk $|a| \le 1 + a_0$.  Moreover we have 
 \[
e_{\bE}(a)=\lim_{n\to\infty}\frac{1}{n} \log \mu_{\bE}\left((J_{\mu_0}T_{\bE}^n)^a\right) \,=\, \lim_{n\to\infty}\frac{1}{n} \log \mu_{0}\left((J_{\mu_0}T_{\bE}^n)^a\right)
\]
and, as a consequence, for  $a \in [-a_0, 1+a_0]$ we have the non-equilibrium steady state fluctuation relation
\begin{equation}
\label{eq:fluct}
e_{\bf E}(a)=e_{\bf E}(1-a)\,.
\end{equation}
\item  The logarithmic moment generating function $e_{\bf E}(a)$ is strictly convex if and only if 
$\log J_{\mu_0}T_{\bE}$ is not a coboundary for some $\psi \in L^2(\mu_{\bE})$, in which case we have 
\[
0 >  e'_{\bE}(0)=\mu_{\bE}(\log J_{\mu_0} T_{\bE})=\eps\mu_0(H)+ o(\eps) \textrm{ (Positivity of entropy production)},
\]
and 
\[
0< e''_{\bE}(0)=\sigma_{\bE}^2=\sigma_{H}^2\eps^2+o(\eps^2) \textrm{ (Positivity of diffusion coefficients)}.
\]
\end{enumerate}
%
%
\end{theorem}

\begin{remark}
The expansion of $\mu_{\bE}(\log J_{\mu_0}T_{\bE})$ in item (3) of Theorem~\ref{thm:moment}
is related to the linear response of the periodic Lorentz gas to the external forces
$\bE = (\bF, \bG)$.  For more explicit relations valid for this class of perturbations,
see \cite{CELSb, CZZ}.
\end{remark}

We prove Theorem~\ref{thm:moment} in Section~\ref{moment proof}.  The main technical elements in 
the proof are first to establish the spectral gap, and then to derive the existence of the relevant
limit(s) and the analyticity of the moment generating function.  The proof of strict convexity
also requires substantial work related to
the Central Limit Theorem.  Once these two properties are established, the fluctuation relation 
\eqref{eq:fluct} follows 
immediately from the transient fluctuation relation, Proposition \ref{ess}.

By using standard large deviation techniques \cite{DZe} we obtain immediately a version of the 
Gallavotti-Cohen fluctuation theorem.

\begin{theorem}
\label{cor:limit theorems}  Under the assumptions of Theorem ~\ref{thm:uniform},
for all $z \in [e_{\bE}'(-a_0), e_{\bE}'(1+a_0)]$, we have
\[
\lim_{\delta \to 0} \lim_{n \to \infty} \frac 1n \log \frac{  \mu_{\bf E} \Big( x  : \frac 1n S_ns (x) \in
[z-\delta, z +\delta] \Big)} {\mu_{\bf E} \Big( x  : \frac 1n S_ns (x) \in
[-z-\delta, -z +\delta] \Big)} = z .
\]
\end{theorem}

The proof is immediate as soon as we recall that the symmetry of the logarithmic moment 
generating function implies the symmetry $I(z)=I(-z)-z$ from \eqref{eq:sym I} for the rate function.


\section{Abstract Framework}
\label{common}

In this section, we present a set of uniform properties {\bf (H1)}-{\bf (H5)} enjoyed by the class
of perturbed billiard maps defined in Section~\ref{model}; these properties guarantee
the Lasota-Yorke inequalities \eqref{eq:uniform LY} with uniform constants.
These conditions are a simplified version of the abstract framework appearing in
\cite{DZ2} since here we consider only finite horizon billiards, so the
technical difficulties associated with the infinite horizon case are excluded.

We also introduce general conditions {\bf (C1)}-{\bf (C4)} to verify that
a perturbation is small
in the sense required for Theorem~\ref{thm:uniform}.  These
conditions are sufficient to establish the framework of \cite{keller
liverani}.  The fact that the specific classes of perturbations we consider in
Section~\ref{model} satisfy {\bf (H1)}-{\bf (H5)} follows from Lemma~\ref{lem:verify H}.


\subsection{A class of maps with uniform properties}
\label{class of maps}

We fix the phase space $M = \cup_{i=1}^d I_i \times [ - \frac{\pi}{2},
\frac{\pi}{2} ]$
of a billiard map
associated with a periodic Lorentz gas as in Section~\ref{results}.
We will denote (normalized) Lebesgue measure on $M$ by $m$,
i.e., $dm = \frac{1}{\pi L} dr d\vf$, where $L = \sum_{i=1}^d |I_i|$.

We define the set $\Si_0 = \{ \vf = \pm \frac{\pi}{2} \}$ and for a fixed
$k_0 \in \mathbb{N}$,
we define for $k \geq k_0$, the homogeneity strips,
\beq\label{homogeneity}
\Ho_k = \{ (r,\vf) : \pi/2 - k^{-2} < \vf < \pi/2 - (k+1)^2 \}.
\eeq
The strips $\Ho_{-k}$ are defined similarly near $\vf = -\pi/2$. We also
define
$\Ho_0 = \{ (r, \vf) : -\pi/2 + k_0^{-2} < \vf < \pi/2 - k_0^{-2} \}$.
The set $\Si_{0,H} = \Si_0 \cup (\cup_{|k| \ge k_0} \partial \Ho_{\pm k}
)$ is therefore fixed
and will give rise to the singularity sets for the maps that we define
below,
i.e. for any map $T$ that we consider, we define
$\Si_{\pm n}^T = \cup_{i = 0}^n T^{\mp i} \Si_0$ to be the singularity
sets for
$T^{\pm n}$, $n \ge 0$.
We assume that $\cS_{\pm n}^T$ comprises finitely many smooth curves for each
$n \in \mathbb{N}$.  We also define the extended singularity sets 
$\Si_{\pm n}^{T, \Ho} = \cup_{i=0}^n T^{\mp i} \Si_{0,H}$ to include the boundaries of the
homogeneity strips.  When the map $T$ is fixed, we sometimes write
$\Si_{\pm n}^{\Ho}$ to simplify notation.

Suppose there exists a class of invertible maps $\mathcal{F}$ such that
for each $T \in \mathcal{F}$, $T : M \setminus \Si_1^T \to M \setminus
\Si_{-1}^T$
is a $C^2$ diffeomorphism on each connected component of $M \setminus
\Si_1^T$.
We assume that elements of $\mathcal{F}$ enjoy the following uniform
properties.

\bigskip
\noindent
{\bf(H1)}
{\bf Hyperbolicity and singularities.}   There exist continuous families
of stable and unstable cones $C^s(x)$ and $C^u(x)$, defined on all of $M$,
which are strictly invariant for the
class $\F$, i.e., $DT(x) C^u(x) \subset C^u(Tx)$ and $DT^{-1}(x) C^s(x)
\subset C^s(T^{-1}x)$
for all $T \in \F$ wherever $DT$ and $DT^{-1}$ are defined.

The cones $C^s(x)$ and $C^u(x)$ are uniformly transverse on $M$
and $\Si_{-n}^T$ is uniformly transverse to $C^s(x)$ for each $n \in \N$
and
all $T \in \F$.   We assume in addition that
$C^s(x)$ is uniformly transverse to
the horizontal and vertical directions on all of $M$.\footnote{This is not
a restrictive
assumption for perturbations of the Lorentz
gas since the standard cones $\hat C^s$ and $\hat C^u$ for the billiard
map satisfy this
property (see for example \cite[Section 4.5]{chernov book}); the common
cones
$C^s(x)$ and $C^u(x)$ shared by all maps in the class $\F$ must therefore
lie
inside $\hat C^s(x)$ and $\hat C^u(x)$ and therefore satisfy this
property.}

Moreover, there exist constants $C_e>0$ and $\Lambda >1$ such that for all
$T \in \F$,
\begin{equation}
\label{eq:uniform hyp}
\| DT^n(x) v \| \ge C_e^{-1} \Lambda^n \| v\|, \forall v \in C^u(x), \; \;
\;
\mbox{and} \; \; \;
\| DT^{-n}(x) v \| \ge C_e^{-1} \Lambda^n \| v\|, \forall v \in C^s(x),
\end{equation}
for all $n \ge 0$, where $\| \cdot \|$ is the Euclidean norm on the
tangent space $\mathcal{T}_xM$.

We also assume a similar unbounded expansion in a neighborhood of $\cS_0$.
We assume there exists $C_c >0$ such that
\begin{equation}
\label{eq:expansion}
C_c [\cos \vf(T^{-1}x)]^{-1} \| v \| \leq \|DT^{-1}(x) v\| \leq C_c^{-1}
 [\cos \vf(T^{-1}x)]^{-1} \| v \| ,\,\,\,\,\,\,\forall x\in M \setminus \cS_{-1}^T, \forall
v\in C^s(x),
\end{equation}
where $\vf(y)$ denotes the angle at the point $y = (r, \vf) \in M$.
Let exp$_x$ denote the exponential map from $\mathcal{T}_xM$ to $M$.
We require the following bound on the second derivative,
\begin{equation}
\label{eq:2 deriv}
C_c [\cos \vf(T^{-1}x)]^{-3} \leq \|D^2T^{-1}(x) v\| \leq C_c^{-1} 
[\cos \vf(T^{-1}x)]^{-3},\,\,\,\,\,\,\forall x\in M \setminus \cS_{-1}^T,
\end{equation}
for all $v \in \mathcal{T}_xM$ such that $T^{-1}(\mbox{exp}_x(v))$ and
$T^{-1}x$ lie in the
same homogeneity strip.\\

\bigskip
\noindent
{\bf(H2)}
{\bf Families of stable and unstable curves.} We call $W$ a {\em stable
curve}
for a map $T \in \F$ if the tangent line to $W$, $\mathcal{T}_xW$ lies in
$C^s(x)$
for all $x \in W$. We call $W$ {\em homogeneous} if $W$ is contained in
one homogeneity
strip $\Ho_k$.   Unstable curves are defined similarly.

\smallskip
\noindent
Let $\widehat\W^s$ denote the set of $\C^2$ homogeneous stable curves in
$M$ whose
curvature is bounded above by a uniform constant $B >0$. We assume there
exists a
choice of $B$ such that $\widehat\W^s$ is invariant under
$\F$ in the following sense:  For any $W \in \widehat\W^s$ and $T \in \F$,
the connected components of $T^{-1}W$ are again elements of $\widehat
\W^s$.
A family of unstable curves $\widehat\W^u$ is defined analogously, with
obvious modifications:
For example, we require the connected components of $TW$ to be elements of
$\widehat\W^u$ for all $W \in \widehat\W^u$ and $T \in \F$.

\vspace{0.15in}

\noindent
{\bf(H3)}
{\bf Complexity bounds (One-step expansion).}\footnote{In \cite{DZ2}, a `weakened one-step
expansion' was also assumed:  $\limsup_{\delta \to 0} \sup_{T \in \cF} \sup_{|W| < \delta}
\sum_i |J_{V_i}T|^\varsigma < \infty$ for some $\varsigma <1$, where the norm
of the Jacobian is measured in the Euclidean norm.  Since here we restrict to finite
horizon, however, this property follows from {\bf (H1)}.}
We assume that there exists an adapted norm $\| \cdot \|_*$,
uniformly equivalent to $\| \cdot \|$, in
which the constant $C_e$ in \eqref{eq:uniform hyp} can be taken to be $1$,
i.e. we have expansion
and contraction in one step in the adapted norm for all maps in the class
$\F$ (for example, the norm from \cite[Sect. 5.10]{chernov book}).

Let $W \in \widehat W^s$.
For any $T \in \F$, we partition the connected components of $T^{-1}W$
into
maximal pieces $V_i = V_i(T)$ such that each $V_i$ is a homogeneous stable
curve in some
$\Ho_k$, $k\geq k_0$, or $\Ho_0$.
Let $|J_{V_i}T|_*$
denote the minimum contraction on $V_i$ under $T$ in the metric
induced by the adapted norm $\| \cdot \|_*$.
We assume that for some choice of $k_0$,
\begin{equation}
\label{eq:step1}
\limsup_{\delta \to 0} \sup_{T \in \F} \sup_{|W|<\delta} \sum_i
|J_{V_i}T|_* < 1,
\end{equation}
where $|W|$ denotes the arclength of $W$.


%

\bigskip
\noindent
{\bf(H4)}
{\bf Bounded distortion.}
There exists a constant $C_d>0$ with the following properties.
Let $W' \in \widehat\W^s$ and for any $T \in \F$, $n \in \N$, let $x, y
\in W$ for some
connected component $W \subset T^{-n}W'$ such that $T^iW$ is a
homogeneous stable curve for each $0 \le i \le n$. Then,
\begin{equation}
\label{eq:distortion stable}
\left| \frac{J_{\mu_0} T^n(x)}{J_{\mu_0} T^n(y)} -1 \right| \; \leq \; C_d
d_W(x,y)^{1/3} \; \;
\mbox{and} \; \; \left| \frac{J_WT^n(x)}{J_WT^n(y)} -1 \right| \; \leq \;
C_d d_W(x,y)^{1/3},
\end{equation}
where as before $J_{\mu_0} T^n$ is the Jacobian of $T^n$ with respect to the smooth
measure
$d\mu_0 = c \cos \vf dr d\vf$.

\smallskip
\noindent
We assume the analogous bound along unstable leaves:
If $W \in \widehat\W^u$ is an unstable curve such that $T^iW$ is a
homogeneous unstable curve
for $0\le i \le n$, then for any $x, y \in W$,
\begin{equation}
\label{eq:D u dist}
\left| \frac{J_{\mu_0} T^n(x)}{J_{\mu_0} T^n(y)} -1 \right| \; \leq \; C_d
d(T^nx,T^ny)^{1/3} .
\end{equation}

\bigskip
\noindent
{\bf(H5)}
{\bf Control of Jacobian.}
Let $\beta, \gamma < 1$ be from the definition of the norms in
Section~\ref{norms} and let
$\theta_*<1$ be from \eqref{eq:one step contract}.
Assume there exists a constant 
$0<\eta < \min \{ \Lambda^\beta, \Lambda^\gamma, \theta_*^{p-1} \}$
such that for any $T \in \F$, 
\[
(J_{\mu_0}T(x))^{-1} \le \eta \qquad \mbox{wherever $J_{\mu_0}T$ is defined.}
\]

Recall the family of stable curves $\widehat \W^s$ defined by {\bf (H2)}.
We define
a subset $\W^s \subset \widehat \W^s$ as follows. By {\bf (H3)} we may
choose
$\delta_0 > 0$ for which there exists $\theta_*<1$ such that
\begin{equation}
\label{eq:one step contract}
\sup_{T \in \F} \sup_{|W| \le \delta_0} \sum_i |J_{V_i}T|_* \le \theta_* .\end{equation}
We shrink $\delta_0$ further if necessary so that the graph transform
argument needed in the proof of
Lemma~\ref{lem:angles}(a) holds.
The set $\W^s$ comprises all those stable curves $W \in \widehat\W^s$ such
that
$|W| \le \delta_0$.



\subsection{Distance in $\F$}
\label{distance}

We define a distance in $\mathcal{F}$ as follows. 
For $T_1, T_2 \in \F$ and $\ve > 0$, let
$N_\ve(\Si^i_{-1})$ denote the $\ve$-neighborhood in $M$ of the
singularity set $\Si^i_{-1}$ of
$T_i^{-1}$, $i = 1,2$.   We say $d_{\F} (T_1, T_2) \le \ve$
if
the maps are close away from their singularity sets in the following
sense:
For $x \notin N_\ve(\Si^1_{-1} \cup \Si^2_{-1})$,

\medskip
\noindent
\parbox{.07 \textwidth}{\bf(C1)}
\parbox[t]{.91 \textwidth}{
$ \displaystyle
d(T_1^{-1}(x) , T_2^{-1}(x))  \le \ve$;
}

\medskip
\noindent
\parbox{.07 \textwidth}{\bf(C2)}
\parbox[t]{.91 \textwidth}{
$ \displaystyle
\left| \frac{J_\mu T_i(x)}{ J_\mu T_j(x)} - 1 \right| \le \ve$, $i,j =
1,2$;
}

\medskip
\noindent
\parbox{.07 \textwidth}{\bf(C3)}
\parbox[t]{.91 \textwidth}{
$ \displaystyle
\left| \frac{J_WT_i(x)}{ J_WT_j(x)}  - 1 \right| \le \ve$,
for any $W \in \W^s$, $i,j = 1,2$, and $x \in W$;
}

\medskip
\noindent
\parbox{.07 \textwidth}{\bf(C4)}
\parbox[t]{.91 \textwidth}{
$ \displaystyle
\| DT_1^{-1}(x) v - DT_2^{-1}(x) v \| \le \sqrt{\ve}$, for any unit vector
$v \in \mathcal{T}_xW$,
$W \in \W^s$.
}

\medskip
We remark that while this notion of distance requires $T_1$ and $T_2$ to be $\cC^1$-close outside
an $\ve$-neighborhood of $\cS_{-1}^1 \cup \cS_{-1}^2$, it does not require $\cS_{-1}^1$ and
$\cS_{-1}^2$ to be close as subsets of $M$.


\section{Definition of the Norms}
\label{norms}

The norms we will use are defined via integration on the set of stable curves
$\W^s$.  Before defining the norms, we define
the notion of a distance $d_{\W^s}(\cdot, \cdot)$ between such curves as
well as
a distance $d_q(\cdot, \cdot)$ defined among functions supported on these
curves.

Due to the transversality condition on the stable cones
$C^s(x)$ given by {\bf (H1)}, each stable curve $W$ can be viewed
as the graph of a function $\vf_W(r)$ of the arc length parameter $r$.
For each $W \in \W^s$,
let $I_W$ denote the interval on which
$\vf_W$ is defined and set $G_W(r) = (r, \vf_W(r))$ to be its graph so
that
$W = \{ G_W(r) : r \in I_W \}$.
We let $m_W$ denote the unnormalized arclength measure on $W$,
defined using the Euclidean metric.

Let $W_1, W_2 \in \W^s$ and identify them with the graphs $G_{W_i}$ of
their
functions $\vf_{W_i}$, $i = 1,2$. Suppose $W_1, W_2$ lie in the same
component of $M$
and let $I_{W_i}$ be the $r$-interval on which each curve is defined.
Denote by $\ell(I_{W_1} \triangle I_{W_2})$ the length of the symmetric
difference
between $I_{W_1}$ and $I_{W_2}$.
Let
$\Ho_{k_i}$ be the homogeneity strip containing $W_i$.
We define the distance between $W_1$ and $W_2$ to be,
\[
d_{\W^s} (W_1,W_2) = \eta(k_1, k_2) +
\ell( I_{W_1} \triangle I_{W_2}) + |\vf_{W_1} -\vf_{W_2}|_{\C^1(I_{W_1}
\cap I_{W_2})}
\]
where $\eta(k_1,k_2) = 0$ if $k_1=k_2$ and $\eta(k_1,k_2) = \infty$
otherwise,
i.e., we only compare curves which lie in the same homogeneity strip.

For $0 \leq \alpha \leq 1$, denote by
$\tilde{\C}^\alpha(W)$ the set of continuous complex-valued functions on
$W$ with H\"{o}lder exponent $\alpha$, measured in the Euclidean
metric, which we denote by $d_W(\cdot, \cdot)$.
We then denote by $\C^\alpha(W)$ the closure of $\C^\infty(W)$
in the $\tilde{\C}^\alpha$-norm\footnote{While $\C^\alpha(W)$ may not contain
all of $\tilde{\C}^\alpha(W)$, it does contain $\C^{\alpha'}\!(W)$ for all $\alpha'>\alpha$.
Defining $\C^\alpha(W)$ in this manner ensures the injectivity of the inclusion
$\cB \hookrightarrow \cB_w$.}:
$| \psi |_{\C^\alpha(W)} = |\psi|_{\C^0(W)} + H^\alpha_W(\psi)$, where
$H^\alpha_W(\psi)$ is the H\"older constant of $\psi$ along $W$.
Notice that with this definition,
$|\psi_1 \psi_2 |_{\C^\alpha(W)} \le |\psi_1|_{\C^\alpha(W)} |\psi_2|_{\C^\alpha(W)}$.
We define $\tilde{\C}^\alpha(M)$ and $\C^\alpha(M)$ similarly.

Given two functions
$\psi_i\in\C^\beta(W_i,\mathbb{C})$, $\beta >0$, we define the distance between
$\psi_1$, $\psi_2$ as
\[
d_\beta(\po,\pt) =|\po\circ G_{W_1}-\pt\circ G_{W_2}|_{\C^\beta(I_{W_1} \cap
I_{W_2})}.
\]
We will define the required Banach spaces by closing $\C^1(M)$ with
respect to
the following set of norms.

Fix $0 < \alpha \le \min \{ \frac 13, \frac{\alpha_1}{2} \} $, where $\alpha_1$ is from
Lemma~\ref{TFJ}.
Given a function $h \in \C^1(M)$, define the \emph{weak norm}
of $h$ by
\begin{equation}
\label{eq:weak}
|h|_w:=\sup_{W\in\W^s}\sup_{\substack{\psi \in\C^\alpha(W)\\
|\psi|_{\C^\alpha(W)} \leq 1}}\int_W h \psi \; dm_W .
\end{equation}
Choose
$\beta$, $\gamma$, $p >0$ such that $\beta < \alpha$, $p\le 1/3$ and
$\gamma < \min \{ p, \alpha-\beta, 1/7 \}$.
We define the \emph{strong stable norm} of $h$ as
\begin{equation}
\label{eq:s-stable}
\|h\|_s:=\sup_{W\in\W^s}\sup_{\substack{\psi\in\C^\beta(W)\\
|W|^p |\psi|_{\C^\beta(W)} \leq 1}}\int_W h \psi \; dm_W
\end{equation}
and the \emph{strong unstable norm} as
\begin{equation}
\label{eq:s-unstable}
\|h\|_u:=\sup_{\varepsilon \leq \varepsilon_0} \; \sup_{\substack{W_1,
W_2 \in \W^s \\
d_{\W^s} (W_1,W_2)\leq \varepsilon}}\;
\sup_{\substack{\psi_i \in \C^\alpha(W_i) \\ |\psi_i|_{\C^\alpha(W_i)}\leq 1\\
d_\beta(\psi_1,\psi_2)
\leq \ve}} \;
\frac{1}{\varepsilon^\gamma} \left| \int_{W_1} h
\psi_1 \; dm_W - \int_{W_2} h \psi_2 \; dm_W \right|
\end{equation}
where $\ve_0 > 0$ is chosen less than $\delta_0$, the maximum length of $W
\in \W^s$ which
is determined by \eqref{eq:one step contract}.
We then define the \emph{strong norm} of $h$ by
\[
\|h\|_\B = \|h\|_s + b \|h\|_u
\]
where $b$ is a small constant chosen in \eqref{eq:LY}.

We define $\B$ to be the completion of $\C^1(M)$ in the strong
norm and $\B_w$ to be the completion of $\C^1(M)$ in the weak norm.
We remark that as a measure,
$h \in \C^1(M)$ is identified with $hd\mu_0$ according to our earlier
convention.
As a consequence, Lebesgue measure $dm = (\cos \vf)^{-1} d\mu_0$ is not
automatically
included in $\B$ since $(\cos \vf)^{-1} \notin \C^1(M)$.  It follows from
\cite[Lemma~5.5]{DZ2} that in fact, $m \in \B$ (and $\B_w$).



\subsection{Properties of the Banach spaces}
\label{recall property}

We recall some properties of our Banach spaces which demonstrate that although they
are spaces of distributions defined as closures of $\C^1$ functions in the stated norms, they
enjoy some natural relations with more familiar spaces of functions and distributions.
Recall $H^\alpha_n(\psi) := \sup_{W \in T^{-n}\cW^s} H^\alpha_W(\psi)$
from Section~\ref{sec:transfer}.

\begin{lemma}
\label{lem:distr} 
The following properties hold.
\begin{itemize}
\item[(i)] (\cite[Lemma 5.4]{DZ2}) There exists $C >0$ such that for any
$h \in \B_w$, $T \in \F$, $n \ge 0$ and
$\psi \in \C^\alpha(T^{-n}\W^s)$,
  \[
  |h(\psi)| \le C |h|_w (|\psi|_\infty + H^\alpha_n(\psi)) .
  \]
  \item[(ii)] (\cite[Lemma 2.1]{DZ3})  There is a sequence of continuous inclusions
$\C^q(M) \hookrightarrow \B \hookrightarrow \B_w \hookrightarrow
(\C^\alpha(M))'$,
for all $q > \gamma/(1-\gamma)$.  The inclusions are injective, except possibly the last.\footnote{
This last inclusion can be made injective by introducting a weight $p'$ in the weak norm similar to
the role of $p$ in the strong stable norm, and requiring that $p' > \alpha$.  This is carried out
in \cite[Lemma 3.8]{DZ3}.}

\item[(iii)]  (\cite[Lemma 3.10]{demers zhang})
The unit ball of $(\B, \| \cdot \|_\B)$ is compactly embedded in $(\B_w, |
\cdot |_w)$.
\end{itemize}
\end{lemma}

We shall need the following result, which is \cite[Lemma 3.5]{DZ3}.
Let $N_\ve(\cdot)$ denote the $\ve$-neighborhood of a set in $M$.

\begin{lemma}
\label{lem:multiplier}
Let $\cP$ be a (mod 0) countable partition of $M$ into open, simply
connected
sets such that:\\
(1) There are constants $K, C_1>0$ such that for each $P \in \cP$ and $W \in \cW^s$, 
$P \cap W$ consists of at most $K$ connected components and for any $\ve >0$,
$m_W(N_\ve(\partial P) \cap W) \le C_1 \ve$;
 (2) Each homogeneity strip $\bH_k$ intersects
at most finitely many $P \in \cP$.

Let $q > \gamma / (1-\gamma)$.
Suppose $f$ is a function on $M$ such that
$\sup_{P \in \cP} |f|_{C^q(P)} < \infty$ and let $h \in \B$.
Then $hf \in \B$ and $$\| hf \|_\B \le C \|h\|_\B \sup_{P \in \cP}
|f|_{C^q(P)}$$
for some uniform constant $C$.
\end{lemma}

We call a potential {\em admissible} for a map $T \in \cF$ if $g$ is at least $1/3$ H\"older 
continuous\footnote{One can decrease the H\"older exponent 1/3 by placing another restriction on 
$\alpha$ and $\gamma$ in the definition of the norms.}  on connected
components of $M \setminus \cS^T_1$:  $\sup_{P \in \cP_1} |g|_{\C^{1/3}(P)} < \infty$, where
$\cP_1$ is the partition of $M$ into connected components of $M \setminus \cS_1^T$.

  Our final lemma of this section shows that $\Lp_{T,g}$ is well-defined as an operator from
$\B$ to $\B$.
Its proof is similar to \cite[Lemma 2.1]{demers zhang}, generalized to include potentials.
\begin{lemma}
\label{lem:L well defined}
If $g$ is an admissible potential for $T$, then $\Lp_{T,g}$ is well-defined as a continuous 
linear operator on both $\B$ and $\B_w$.
\end{lemma}

\begin{proof}
Let $h \in \C^1(M)$. The Lasota-Yorke inequalities of Proposition~\ref{prop:ly} show that
$\Lp_{T,g} h$ has finite norm in both $\B$ and $\B_w$.  In order to show that
$\Lp_{T,g} h$ belongs to $\B$, 
we must approximate
$\Lp_{T,g} h$ by $\C^1$ functions in the norm $\| \cdot \|_\B$.
Note that
$\Lp_{T,g} h$ has a countable number of smooth discontinuity curves given by
$\Si_{-1}^{\bH}$
(we include the images of boundaries of the homogeneity strips). These
curves define a countable partition $\cP$ of $M$ into open simply
connected sets, and each $\bH_k$ can intersect countably many
$P \in \cP$.  In addition, the $\C^1$ norm of
$\Lp_{T,g} h$ blows up near
the curves $T\Si_0$.

For $j \ge k_0$ let $P^j$ denote an element of $\cP$ such that $T^{-1}P^j
\subseteq \bH_j$.
Again, the labeling is not unique, but for each $j$, the number of
elements in $\cP$ which are
assigned the label $j$ is finite (even in the infinite horizon case).
Let $P^J = \cup_{j > J} P^j$.
We claim that $\| \Lp_{T,g} h |_{P^J}  \|_\B$ is arbitrarily small
for $J$ sufficiently large. On the finite set of $P^j$ with $j \le J$, the
$\C^1$ norm of
$\Lp_{T,g} h$ is finite and the modified partition $\cP^* = \{ P^j \}_{j \le
J} \cup \{ P^J \}$
satisfies the requirements of Lemma~\ref{lem:multiplier}.
So we may approximate $\Lp_{T,g} h$ using Lemma~\ref{lem:multiplier}
on $M \setminus P^J$ and approximate $\Lp_{T,g} h$ by 0 on $P^J$. Thus the
lemma
follows once we establish our claim.

Indeed, the claim is trivial using the estimates contained in Appendix~\ref{appendix}.
For example, we must estimate
$\| (\Lp_{T,g} h) |_{P^J} \|_s = \| 1_{P^J} \Lp_{T,g} h \|_s$.
Taking $W \in \W^s$ and $\psi \in \C^\beta(W)$ with $|W|^p|\psi|_{\C^\beta(W)} \le
1$, we write
\[
\int_W 1_{P^J} \Lp_{T,g} h \, \psi \, dm_W
= \int_{T^{-1} (W \cap P^J)}h (J_{\mu_0}T)^{-1} e^{S_n g} J_{T^{-1}W}T \, \psi \circ T
\, dm_W,
\]
and the homogeneous stable components of $T^{-1}(W \cap P^J)$ correspond
precisely to
the tail of the series considered in \eqref{eq:weak estimate} and
following and so can be
made arbitrarily small by choosing $J$ large (notice that we do not need
contraction here
so that we may use the simpler estimate similar to Section~\ref{weak norm}
applied to the
strong stable norm rather than the estimate of Section~\ref{stable norm}.

Similarly, in estimating $\| \Lp_{T,a} h \|_u$, one can see that the
contribution from
$P^J$ corresponds to the tail of the series from the estimates of
Section~\ref{unstable norm},
and so this too can be made arbitrarily small by choosing $J$ large.
\end{proof}


\section{Proof of Theorem~\ref{thm:uniform}}
\label{uniform}

The proof of Theorem~\ref{thm:uniform} relies on the following
more general proposition.
Recall that an admissible potential $g$  for $T \in \cF$ is one that satisfies
$\sup_{P \in \cP_1} |g|_{\C^{1/3}(P)} < \infty$, where
$\cP_1$ is the partition of $M$ into connected components of $M \setminus \cS_1^T$.
For an admissible potential $g$, define $C_g := |g|_{\C^\alpha(M)} e^{|g|_{\C^\alpha(M)} \delta_0^\alpha }$.

\begin{proposition}
\label{prop:ly}
There exists $C>0$, depending only on {\bf (H1)}-{\bf(H5)}, such that for
any
$T \in \F$, admissible potential $g$, $h \in \B$ and $n \ge 0$,
\begin{eqnarray}
|\Lp_{T,g}^n h|_w & \le & C C_g |(J_{\mu_0}T^n)^{-1} e^{S_ng}|_{\infty} |h|_w,
\label{eq:weak norm} \\
\| \Lp_{T,g}^n h \|_s & \le &CC_g |(J_{\mu_0}T^n)^{-1} e^{S_ng}|_{\infty} \left( (
\theta_*^{(1-p)n} + \Lambda^{-\beta n}) \|h\|_s + C
\delta_0^{-p}|h|_w\right),
\label{eq:stable norm} \\
\| \Lp_{T,g}^n h \|_u &\le & CC_g |(J_{\mu_0}T^n)^{-1} e^{S_ng}|_{\infty} \left(
\Lambda^{-\gamma n} \| h\|_u + C C_3^n \|h \|_s\right),
\label{eq:unstable norm}
\end{eqnarray}
where $C_3$ is from Lemma~\ref{lem:growth}(d).
\end{proposition}
The proof of this proposition is fairly technical, but has a lot of
similarity with the corresponding inequalities proved in \cite{demers zhang}
and \cite{DZ2} in the case $g=0$. 
We put the proof in Appendix~\ref{appendix} for completeness and to
draw out the explicit dependence on the added potential.

Let $\sigma := \max \{ \theta_*^{1-p}$, $\Lambda^{-\beta},
\Lambda^{-\gamma} \} < 1$,
and choose $N \ge 0$ such that
\begin{equation}
\label{eq:LY}
\begin{split}
\| \Lp_{T,g}^N h \|_\B &  = \| \Lp^N_{T,g} h \|_s + b \| \Lp_{T,g}^N h \|_u \\
& \le C_g |(J_{\mu_0}T^N)^{-1} e^{S_Ng}|_{\infty} \left(\frac{\sigma^N}{2} \| h
\|_s + C \delta_0^{-p} |h|_w + b \sigma^N \| h \|_u + b C C_3^N
\|h\|_s\right) \\
& \leq C_g |(J_{\mu_0}T^N)^{-1} e^{S_Ng}|_{\infty} \left(\sigma^N \| h \|_\B +
C_{\delta_0} |h|_w\right),
\end{split}
\end{equation}
providing $b$ is chosen sufficiently small with respect to $N$.  This is the standard
Lasota-Yorke inequality for $\Lp_{T,g}$ for a general potential $g$.
In order to specialize to the case $g = a \log J_{\mu_0}T$, we recall the following
lemma about the form of the Jacobian $J_{\mu_0}T$ derived in \cite{CZZ}.

\begin{lemma}{\em (\cite[Lemmas 3.2 and 4.2]{CZZ})}
\label{TFJ}
Fix $\ve, \tau_*$ and $C_0$ and consider $T_{\bF, \bG} \in \F(\ve, \tau_*, C_0)$.

First, assume there is no twist force $\bG=0$ and denote $T_{\bF, \mathbf{0}} =T_{\bF}$.  
Then the Jacobian of $T_{\bF}$ with respect
to $\mu_0$ is given by 
\beq
\label{cJeps} 
J_{\mu_0} T_{\bF} =\exp \left( \int_0^{\tau_{\bF}( \bx)}p \, \tfrac{\partial \kappa}{\partial \theta} \, dt \right), 
\eeq
where $\tau_{\bF}$ is the free path for the system $T_{\bF}$ and $\kappa$ is from \eqref{eq:formh}.

Next, assume $\bG \neq 0$.  Then by Assumption ({\bf A4}) the Jacobian of $T_\bE = T_{\bF, \bG}$ satisfies,
\beq
\label{eq:JFG}
J_{\mu_0} T_{\bE} = J_{\mu_0} \bG(T_{\bF} ) J_{\mu_0}T_{\bF}.
\eeq 
Moreover, we may write,
\beq
\label{eq:eps jac}
J_{\mu_0} T_{\bE} = 1 + \ve H, \qquad \mbox{where} \quad H = \tfrac 1\ve \left( J_{\mu_0}T_{\bE}-1\right),
\eeq
$|H|_{\infty} \le C_H$ for some $C_H>0$ independent of $\ve$ and $H$ is $C^{\alpha_1}$ for some\footnote{The restriction on $\alpha$ comes from the fact that $H$ is at least $C^{\alpha_0}$
for some $\alpha_0 > 1/3$ by Assumption {\bf (A3)}, but in general not smoother than
$\tau_{\bF}$, which is only $1/2$-H\"older continuous.} 
$\alpha_1 \in (1/3, 1/2]$
on each component of $\Si^{T_\bE}_1$.
\end{lemma}

Writing the twist map $\bG$ in $(r, \vf)$ coordinates, we have
$\bG(r,\vf) = (r,\vf) + (G^1(r,\vf), G^2(r,\vf))$, where $G^1$ and $G^2$ are smooth functions of $r$ and $\vf$.
Since by {\bf (A3)} the singularity sets of $T_{\bF, \bG}$ and $T_{\bF}$ are the same, 
we may write
\[
J_{\mu_0}\bG(T_{\bF}\bx) = 1 + \frac{\partial G^1}{\partial r}(T_{\bF}\bx) + \frac{\partial G^2}{\partial s}(T_{\bF}x) + \frac{\partial G^1}{\partial r}(T_{\bF}\bx) \frac{\partial G^2}{\partial s}(T_{\bF}\bx)
- \frac{\partial G^1}{\partial s}(T_{\bF}\bx) \frac{\partial G^2}{\partial r}(T_{\bF}\bx),
\]
whenever $\bx = (r, \vf) \notin \Si^{T_{\bF}}_1$.  Again using {\bf (A3)}, we have that the
$|\frac{\partial G^i}{\partial r}|$ and $|\frac{\partial G^i}{\partial s}|$ are bounded by $\ve$.

This representation combined with \eqref{cJeps} verifies that the constant
$C_H$ given by Lemma~\ref{TFJ} can be chosen independently of $\ve$.

\subsection{A Spectral Gap for $\Lp_{T,a}$}
\label{sprectral gap}

Now we fix $a_0 >0$ and the interval
$[-a_0, 1+a_0]$ as in the statement of Theorem~\ref{thm:uniform}.
Due to Lemma~\ref{TFJ}, we may choose $\eps_0 >0$ so small that
for all $a \in [-a_0, 1+a_0]$,
\beq
\label{eq:eps_0}
\frac{(1-\sign(a-1)C_H\eps_0)^{a-1}}{(1+\sign(a-1)C_H\eps_0)^{a-1}}>\sigma,
\eeq
where $\sigma$ is from \eqref{eq:LY}.

The next lemma establishes the quasi-compactness of $\Lp_{T,a}$.

\begin{lemma}
\label{lemma:quasi}
Let $a \in [-a_0, 1+a_0]$ and $\ve_0$ be as chosen in \eqref{eq:eps_0}.
Then for all $T \in \F(\ve_0, \tau_*, C_0)$,
$\Lp_{T, a}$ is quasi-compact as an operator on $\B$.
\end{lemma}

\begin{proof} 
When $g = a\log J_{\mu_0}T$, we have $(J_{\mu_0}T^N)^{-1} e^{S_Ng} = (J_{\mu_0}T^N)^{a-1}$
and so \eqref{eq:LY} together with Lemma~\ref{TFJ} yield
the required inequality \eqref{eq:uniform LY}
for Theorem~\ref{thm:uniform}.  Due to the compactness of the
unit ball of $\B$ in $\B_w$ \cite[Lemma
3.10]{demers zhang},
this implies the essential spectral radius 
of $\Lp_{T,a}$, $\rho_{\mbox{\scriptsize ess}}(\Lp_{T,a})$ is at most 
$\sigma (1+ \sign(a-1)C_H\eps_0)^{a-1}$.
To prove that $\Lp_{T,a}$ is quasi-compact, it remains to show that the spectral radius of 
$\Lp_{T, a}$, $\rho(\Lp_{T,a})$, is strictly larger than $\rho_{\mbox{\scriptsize ess}}(\Lp_{T,a})$.

To obtain a lower bound on $\rho(\Lp_{T,a})$,
note that
$$
\rho(\Lp_{T,a})=\lim_{n\to\infty}\|\Lp_{T,a}^n\|_{\B}^{1/n}\geq
\lim_{n\to\infty}\|\Lp_{T,a}^n 1\|_{s}^{1/n}.
$$
Then we have
\begin{align*}
\|\Lp_{T,a}^n 1\|_{s}&=\sup_{W\in\W^s}\sup_{\substack{\psi\in\C^\beta(W)\\
|W|^p |\psi|_{\C^\beta(W)} \leq 1}}\int_W \Lp_{T,a}^n 1 \cdot\psi \; dm_W\\
&\geq \sup_{W\in \W^s}\sup_{\substack{\psi\in\C^\beta(W)\\
|W|^p |\psi|_{\C^\beta(W)}\leq 1}}\inf \Lp_{T,a}^n 1 \int_W\psi\,dm_W
\; \geq \; \inf (1+\eps_0 H)^{(a-1)n} \, \| 1\|_s,
\end{align*}
using Lemma~\ref{TFJ} and the identity $\Lp_{T,a}^n 1 = (J_{\mu_0}T^n)^{a-1} \circ T^{-n}$.
This implies that
$$
\rho(\Lp_{T,a})=\lim_{n\to\infty} \|\Lp_{T,a}^n\|^{\frac{1}{n}}\geq
(1-\text{sign}(a-1)C_H \eps_0)^{a-1}.
$$
Combining this with the upper bound on the 
essential spectrum of $\Lp_{T,a}$ and the choice of $\eps_0$ from \eqref{eq:eps_0}, we conclude
\[
\rho_{\mbox{\scriptsize ess}}(\Lp_{T,a})
 \le \sigma (1+\sign(a-1)C_H\eps_0)^{a-1}
< (1-\sign(a-1)C_H\eps_0)^{a-1} \le \rho(\Lp_{T,a}).
\]
\end{proof}

Recall from Section~\ref{recall property} that a function $g: M \to \mathbb{R}$ is an admissible
potential for $T \in \cF$ if $|g|_{\C^{1/3}(\cP_1)} := \sup_{P \in \cP_1} |g|_{\C^{1/3}(P)} < \infty$, where
$\cP_1$ is the partition of $M$ into connected components of $M \setminus \cS_1^T$.

\begin{lemma}
\label{lem:analytic}
Suppose $g$ is an admissible potential for $T \in \cF(\ve_0, \tau_\ast, C_0)$.
Then the map $z \mapsto \Lp_{T,zg}$ is analytic for all $z\in \mathbb{C}$.
\end{lemma}

\begin{proof}
Define the operator $\cA_n h = \Lp_T(g^nh) = g^n \circ T^{-1} \Lp_T h$, for $h \in \B$.  
Notice that since $g$ is H\"older continuous on elements of $\cP_1$, it follows that
$g \circ T^{-1}$ is H\"older continuous on elements of $\cP_{-1}$, the partition of $M$
into connected components of $M \setminus \cS_{-1}^T$.  Since $\cS_{-1}^T$ consists of
finitely many
curves that are uniformly transverse to the stable cone, we claim that 
$g \circ T^{-1}$ satisfies the assumptions of Lemma~\ref{lem:multiplier}.
Indeed, we have the following estimate for the H\"older regularity of $g \circ T^{-1}$.
For any $x, y$ in the same component of $\cS_{-1}^T$,
\begin{equation}
\label{eq:holder}
\frac{|g \circ T^{-1}(x) - g \circ T^{-1}(y)|}{d(x,y)^{1/6}} = 
\frac{|g \circ T^{-1}(x) - g \circ T^{-1}(y)|}{d(T^{-1}(x), T^{-1}(y))^{1/3}} \frac{d(T^{-1}(x), T^{-1}(y))^{1/3}}{d(x,y)^{1/6}} .
\end{equation}
The first factor is bounded by $|g|_{\C^{1/3}(\cP_1)}$, while the second factor is uniformly bounded
due to the fact that $|T^{-1}W| \le C |W|^{1/2}$ for any $W \in \cW^s$ by {\bf (H3)}
(see, for example, \cite[Exercise~4.50]{chernov book}).  Thus $g \circ T^{-1}$ is $1/6$-H\"older
continuous on $\cP_1$ and $1/6 \ge \gamma/(1-\gamma)$ since $\gamma < 1/7$
so that $g \circ T^{-1}$
satisfies the conditions of Lemma~\ref{lem:multiplier}.

Now 
Lemma~\ref{lem:multiplier} implies that $g^n \circ T^{-1} \Lp_T h \in \B$ and moreover,
\[
\| \cA_n h \|_\B = \| g^n \circ T^{-1} \Lp_T  h \|_\B \le C \| \Lp_T h\|_\B |g^n \circ T^{-1}|_{\C^{1/6}(\cP_{-1})}
\le C \| h \|_\B 
|g|^n_{C^{1/3}(\cP_1)} ,
\]
where we used \eqref{eq:holder} along with the simple fact that
$|fg|_{C^q}\le |f|_{\C^q} |g|_{\C^q}$ to estimate
$|g^n|_{\C^q} \le |g|_{\C^q}^n$.

Therefore, the operator $\sum_{n=0}^\infty \frac{z^n}{n!} \cA_n$ is well defined
on $\B$
and equals $\Lp_{T, zg}$ since once we know the sum converges,
\[
\sum_{n=0}^\infty \frac{z^n}{n!} \cA_n h(\psi) = h \left(
\sum_{n=0}^\infty \frac{z^n}{n!}
g^n \cdot \psi \circ T \right) = h(e^{zg} \psi \circ T) = \Lp_{T, zg}
h(\psi),
\;\; \; \mbox{for } \psi \in \C^\alpha(\W^s) .
\]
\end{proof}

With the analyticity of $z \mapsto \Lp_{zg}$ established, it follows from
analytic
perturbation theory \cite{kato} that both the discrete
spectrum and the corresponding spectral projectors of $\Lp_{T,zg}$ vary
smoothly with $z$. 
We will use the smooth dependence of the spectrum on $z$ to prove that
$\Lp_{T,a}$ has a spectral gap.

\begin{lemma}
\label{lem:peripheral}
Fix $a_0, \tau_*, C_0 >0$ and let $\ve_0$ be as in \eqref{eq:eps_0}.  
Then there exists $0 < \ve_1 \le \ve_0$ such that for all 
$T = T_{\bE} \in \cF(\ve_1, \tau_*, C_0)$,
$\Lp_{T,a}$ has a simple eigenvalue $\lambda=\rho(\Lp_{T,a})$ and all
other eigenvalues have modulus strictly smaller than $\lambda$, i.e.~$\Lp_{T, a}$
has a spectral gap as an operator on $\B$.
\end{lemma}

\begin{proof}
Fix a uniform family $\cF(\ve_0, \tau_*, C_0)$ satisfying {\bf (H1)}-{\bf (H5)} 
and \eqref{eq:eps_0}
such that
$\Lp_{T}$ has a spectral gap for all $T \in \cF(\ve_0, \tau_*, C_0)$ by \cite{DZ2}.

Fixing $T = T_{\bE} \in \cF(\ve_0, \tau_\ast, C_0)$ and using Lemma~\ref{TFJ}, 
we know $-s$ is an admissible potential for $T$.  
According to Lemma~\ref{lem:analytic}, the derivative,$\frac{d}{dz} \Lp_{T, -zs} = \sum_{n \ge 1} \frac{z^{n-1}}{(n-1)!} \cA_n$ is well-defined as a bounded linear operator on $\B$, and
\[
\| \frac{d}{dz} \Lp_{T, -zs} \|_{\B} \le C \| \Lp_T \|_{\B} |s|_{\C^{1/3}(\cP_1)} e^{|z| |s|_{\C^{1/3}(\cP_1)}},
\]
for a uniform constant $C$ (depending only on $\cF$).  Thus for any $a \in [-a_0, 1+a_0]$,
\begin{equation}
\label{eq:close}
\| \Lp_T - \Lp_{T, a} \|_{\B} \le C \| \Lp_T \|_{\B} |s|_{\C^{1/3}(\cP_1)} e^{|a| |s|_{\C^{1/3}(\cP_1)}} |a|
\le C |a| \| \Lp_T \|_{\B} | \log (1+ \ve H) |_{\C^{1/3}(\cP_1)} e^{|\log (1+ \ve H)|_{\C^{1/3}(\cP_1)}^{|a|}},
\end{equation}
where we have used Lemma~\ref{TFJ} and $\ve$ is the optimal $\ve$ for $\bE$.  

It follows from \cite{DZ2} that the spectrum of $\Lp_{T_\bE}$ varies continuously in $\bE$
and converges to the spectrum $\Lp_{T_0}$ as $\bE$ shrinks to 0 (in $\C^1$ norm).  Thus
there exists $0 < \ve_2$, $\ve_2 \le \ve_0$ such that all $T_{\bE} \in \cF(\ve_2, \tau_*, C_0)$
enjoy a uniform spectral gap, i.e., the distance between 1 and the second largest eigenvalue
of $\Lp_{T_\bE}$ is bounded below by a uniform constant; call this constant $\delta >0$.
Then by \eqref{eq:close}, there exists $\ve_1>0$ such that
$\Lp_{T, a}$ has a spectral gap for any $T_{\bE} \in \cF(\ve_1, \tau_*, C_0)$ and
all $a \in [-a_0, 1+a_0]$.
\end{proof}

We will also find it convenient to have the following continuity in $\ve$.

\begin{lemma}
\label{lem:cont ve}
Fix $a \in [-a_0, 1+a_0]$ and $T_0 = T_{\mathbf{0}, \mathbf{0}} \in \cF(\ve_1, \tau_*, C_0)$.
There exists $C>0$ such that for all $\ve \le \ve_1$ and all $T_{\bE} \in \cF(\ve_1, \tau_*, C_0)$ with
$d_{\cF}(T_0, T_{\bE}) \le \ve$, we have 
\[
\sup \{ | \Lp_{T_{\bE}, a} h - \Lp_{T_0, a} h |_w : \|h\|_{\cB} \le 1 \} \le C \ve^{\gamma/2} .
\]
This implies in particular that the leading eigenvalue and associated spectral projectors of 
$\Lp_{T_{\bE},a}$ vary continuously with $\bE$ in the $\ve_1$ neighborhood of $T_0$. 
\end{lemma}

\begin{proof}
The proof is essentially the same as the proof of \cite[Theorem 2.3]{demers zhang},
except with the added potential $(J_{\mu_0}T)^{a-1}$.  We just sketch the proof here, noting
the necessary additions.

Fixing $T_0$ and $T_{\bE}$ as in the statement
of the lemma, we choose $h \in \C^1(M)$ with $\| h \|_{\cB} \le 1$ and $W \in \cW^s$.  Let 
$\psi \in \C^\alpha(W)$ satisfy $|\psi|_{\C^\alpha(W)} \le 1$.  For the weak norm of the difference, 
we must estimate
\[
\int_W (\Lp_{T_0, a} h - \Lp_{T_{\bE}, a} h) \psi \, dm_W 
= \int_{T_0^{-1}W} h J_{T_0^{-1}W} T \, \psi \circ T_0
- \int_{T_{\bE}^{-1}W} h (J_{\mu_0}T_{\bE})^{a-1} J_{T_{\bE}^{-1}W} T \, \psi \circ T_{\bE} ,
\]
where we have used the fact that $J_{\mu_0}T_0 =1$.
The required estimate is similar to the estimate for the strong unstable norm in contained
in Section~\ref{unstable norm}, except that we have one stable curve iterated under two different maps
instead of two close stable curves iterated under the same map.  However, the decomposition
is the same, we subdivide $T_0^{-1}W$ and $T_{\bE}^{-1}W$ into matched and unmatched
pieces.  The matched pieces can be connected by a transverse foliation of unstable curves,
while the unmatched pieces are short.
The estimates proceed precisely as in 
\cite[Section 5]{DZ2}, with \cite[Lemma 5.1]{DZ2} providing the
bounds on all the relevant quantities.  The only additional piece in the present estimate
is the presence of the potential $(J_{\mu_0}T_{\bE})^{a-1}$.  

For the sum over unmatched unmatched pieces, it is bounded by 
$C \ve^{\gamma/2} |(J_{\mu_0}T_{\bE})^{a-1}|_{\C^\beta(\cP_1)} \| h\|_s$ using the strong stable norm precisely as in
\eqref{eq:first unstable}, with $n=1$, since each unmatched piece has length at most $\ve^{1/2}$.

Suppose $U_1$ and $U_2$ are two matched pieces of $T_0^{-1}W$ and $T_{\bE}^{-1}W$,
respectively.  By construction, they are defined over a common $r$-interval $I$, i.e. they
can be written as graphs of functions 
\[
U_j = G_{U_j}(I) = \{ (r, \vf_{U_j}(r)) : r \in I \}
\]
and $d_{\cW^s}(U_1, U_2) \le C \ve^{1/2}$ (\cite[Lemma 5.1(a)]{DZ2}).  
The estimate over matched pieces proceeds precisely as in \eqref{eq:stepone} and the only
difference in test functions unaccounted for in \eqref{eq:diff}
is $|(J_{\mu_0}T_{\bE})^{a-1} -1|_{\C^\beta(U_2)}$.
We will show that
\begin{equation}
\label{eq:a}
|(J_{\mu_0}T_{\bE})^{a-1} - 1|_{\C^\beta(U_2)} \le C \ve^{1-3\beta},
\end{equation}
for some uniform constant $C$ depending on $a_0$. 
Indeed $|(J_{\mu_0}T_{\bE})^{a-1} - 1|_{\C^0(U_2)} \le C|a-1|\ve$ follows from Lemma~\ref{TFJ}.
For the H\"older constant, we take $x, y \in U_2$ and estimate on the one hand using {\bf (H4)},
\[
|(J_{\mu_0}T_{\bE})^{a-1}(x) - (J_{\mu_0}T_{\bE})^{a-1}(y)|
\le C |(J_{\mu_0}T_{\bE})^{a-1}|_{\C^0(U_2)} d(x,y)^{1/3} \le C' d(x,y)^{1/3}.
\] 
While on the other hand,
\[
|(J_{\mu_0}T_{\bE})^{a-1}(x) - (J_{\mu_0}T_{\bE})^{a-1}(y)| \le C'' \ve,
\] 
using Lemma~\ref{TFJ} once again.  So the H\"older constant is bounded by the minimum
of these two expressions,
\[
 \min \{ C'  d(x,y)^{1/3 - \beta}, C'' \ve d(x,y)^{-\beta}  \}.
\]
This bound can be no worse than when the two quantities are equal, i.e. $\ve = (C'/C'') d(x,y)^{1/3}$.
Thus $H^\beta_W((J_{\mu_0}T_{\bE})^{a-1} - 1) \le C \ve^{1-3\beta}$.
This proves \eqref{eq:a}.

Now gathering terms over matched and unmatched pieces as in \eqref{eq:final unstable} 
or \cite[Eq. (5.9)]{DZ2}, we see that the
least power of $\ve$ is $\ve^{\gamma/2}$, from the unmatched pieces 
(notice that $\ve^{1-3\beta} < \ve^{\frac 13 - \beta}$ and 
$\gamma \le \alpha - \beta \le \frac 13 - \beta$).  This completes the proof of the lemma.
\end{proof}

Fix $T \in \cF(\ve_1, \tau_*, C_0)$.
Let $h_{a}\in \B$ be the eigenvector of $\Lp_{T,a}$ corresponding to the eigenvalue
$\lambda_a$ of maximum modulus, and  $\nu_a\in \B^*$
be the corresponding eigenvector of the dual $\Lp_{T,a}^\ast$.
That is, $\Lp_{T,a} h_{a}=\lambda_a h_{a}$, and $\Lp_a^*\nu_a=\lambda_a\nu_a$.
Due to the spectral gap for $\Lp_{T,a}$, we have the following spectral decomposition,
\begin{equation}
\label{eq:decomp}
\Lp_{T,a}^n h = \lambda_a^n \Pi_a h + R_a^n h,
\end{equation}
where $\Pi_a R _a = R_a \Pi_a = 0$ and the spectral radius of $R_a$ is strictly smaller than
$\lambda_a$.  Also, for any $h \in \B$, $\Pi_a h = c_a(h) h_a$, where $c_a : \B \to \mathbb{R}$
is a bounded linear functional.  Notice that $\lambda_a$ must be real since 
$\Lp_{T,a}$ is a real operator and the spectral gap for $\Lp_{T,a}$ is obtained as a 
perturbation of $\Lp_{T,0}$, which has $\lambda_0 = 1$.

The following lemma completes the proof of Theorem~\ref{thm:uniform}.

\begin{lemma}
Both eigenvectors $h_a$ and $\nu_a$ are positive measures.
Moreover, the pairing
$\mu_a:=h_a\otimes\nu_a$ defines an invariant measure for $T$.
\end{lemma}
\begin{proof}
Due to \eqref{eq:decomp}, for any $\psi\in \C^\alpha(M)$,
$$
|c_a(1) h_a(\psi)| = \lim_{n \to \infty} |\lambda_{a}^{-n} \Lp^n_{T,a} 1 (\psi) |
\le \lim_{n \to \infty} |\psi|_\infty |\lambda_a^{-n} \Lp^n_{T,a} 1 (1) | = |\psi|_\infty |c_a(1)| |h_a(1)| .
$$
Now $c_0(1) = 1$ and $c_a(1)$ is continuous in $a$ by Lemma~\ref{lem:analytic},
so by \eqref{eq:close}, $c_a(1) >0$ for $\ve \in [0, \ve_1]$. 
This, together with the above estimate, implies that
$h_{a}$ is a measure. Then it is evident that $h_a$ is a positive measure due to the
positivity of $\Lp_{T,a}$.

Similarly, one can show that $\nu_a$ is also
a positive measure since 
$$\lim_{n \to \infty} \lambda_{a}^{-n} (\Lp_{T,a}^*)^n 1
(\psi) = c_a^*(1) \nu_{a}(\psi),$$
for some linear functional $c_a^*$.

By Lemma~\ref{lem:multiplier}, if $\psi$ is a piecewise H\"older continuous function on $M$,
then $\psi h_a \in \B$.  So we may define a measure on $M$ via
the pairing $\mu_{T,a}:=h_a\otimes \nu_a$, i.e.
$\mu_{T,a}(\psi) = \langle \psi h_a, \nu_a \rangle$, where $\langle \cdot, \cdot \rangle$
denotes the pairing between $\B$ and its dual. 
Moreover, the measure $\mu_{T,a}$  is invariant under $T$:
\begin{align*}
\mu_{T,a}(\psi\circ T)&=  \langle \psi \circ T \cdot h_a , \lambda_a^{-1} \Lp_{T, a}^*\nu_a \rangle
= \lambda_a^{-1} \langle \Lp_{T,a}(\psi\circ T \cdot h_a) , \nu_a \rangle \\
&=\lambda_a^{-1}\langle \psi \Lp_{T,a}(h_a) , \nu_a \rangle = \langle \psi h_a, \nu_a \rangle
=\mu_{T,a}(\psi), \qquad \mbox{for any $\psi \in C^p(M)$},
\end{align*}
where we have used Lemma~\ref{lem:multiplier} to conclude that $\psi \circ T \cdot h_a \in \B$.
\end{proof}

\begin{remark}
\label{rem:positive}
Notice that when $a=0$, the smooth measure 
$\mu_0$ is the conformal measure with respect to $\Lp_{T,0}$, i.e. $\Lp_{T,0}^* \mu_0 = \mu_0$,
so that $\nu_0 = \mu_0$ and
$\langle h_0, \mu_0 \rangle = 1$.  It then follows from Lemma~\ref{lem:analytic}
and \eqref{eq:close} that
that we may choose $\ve_1>0$ sufficently small so that
$\langle h_a, \mu_0 \rangle >0$ and $c_a(h_0)>0$ for all $a \in [-a_0, 1+a_0]$.
\end{remark}



\section{Proof of Theorem~\ref{thm:moment}}
\label{moment proof}

In this section, we shall be more explicit about the dependence of the various objects
on the forces $\bE = (\bF, \bG)$.
We shall use the following notation for the map $T = T_{\bE}$ and the potential $e^{a g_0}$.
We have the following decomposition according to \eqref{eq:decomp}:
\[
\Lp_{T_{\bE}, a} = \lambda_{\bE, a} \Pi_{\bE, a} + R_{\bE, a} .
\]
Denote by $\mu_{\bE, a} =h_{\bE, a} \otimes \nu_{\bE, a}$  the $T_{\bE}$-invariant 
measure constructed using the left and right eigenvectors of $\Lp_{T_{\bE}, a}$.  When $a = 0$, 
in what follows, we will drop the subscript corresponding to $a$, and simply write
$\mu_{\bE} = h_{\bE} \otimes \nu_{\bE}$ for
the SRB measure of the perturbed system 
$T_\bE \in \F(\ve_1, \tau_*, C_0)$. 
 Note this notation is consistent with our
use of $\mu_0 = \mu_{\mathbf{0}, 0}$ as both the conformal measure for $\Lp_{T,0}$ as well as
the smooth invariant measure 
corresponding to the classical billiard map $T_{\mathbf{0}, \mathbf{0}}$, with
$\bF = \bG = \mathbf{0}$.  Indeed, when $\bE = \mathbf{(0,0)}$, then $h_{\mathbf{0}, 0} = 1$.

The moment generating function $e_{\bE}(a)$ is defined as in \eqref{eq:moment},
\begin{align*}
e_{\bE}(a)&=\lim_{n\to\infty}\frac{1}{n} \log \mu_{\bE} ((J_{\mu_0}T_\bE^n)^{a}).
\end{align*}


\begin{proof}[Proof of Theorem~\ref{thm:moment}]
The existence and uniqueness of $\mu_{\bE}$ for Item (1) follow from the spectral gap
of $\Lp_{T_{\bE}}$ established by Theorem~\ref{thm:uniform}.

To prove item (2), first recall that $\langle h_{\bE, a}, \mu_0 \rangle > 0$ and $c_{\bE, a}(h_{\bE, 0}) > 0$ 
by choice of $\ve_1$ and Remark~\ref{rem:positive}.  Then
\begin{align*}
e_\bE (a)&=\lim_{n\to\infty}\frac{1}{n} \log \mu_\bE ((J_{\mu_0}T_\bE^n)^{a})
=\lim_{n\to\infty}\frac{1}{n} \log \langle h_{\bE} \cdot (J_{\mu_0}T_\bE^n)^{a}, \mu_0 \rangle
=\lim_{n\to\infty}\frac{1}{n}\log \langle \Lp_{T_\bE, a}^n h_{\bE} , \mu_0 \rangle \\
&=\lim_{n\to\infty}\frac{1}{n}\log \langle \lambda_{\bE, a}^{n}c_{\bE, a}(h_{\bE})
h_{\bE, a} + R^n_{\bE, a} h_{\bE}, \mu_0 \rangle= \log \lambda_{\bE, a}.
\end{align*}
Thus by Lemma~\ref{lem:analytic}, since $\lambda_{\bE, a}$ is simple, $e_{\bE}(a)$
is analytic as a function of $a$ for $a \in [-a_0, 1+a_0]$.

Now let  $\nu \in \B$ be a probability measure with $c_{\bE, a}(\nu)>0$.  Then the 
limit 
\[
\lim_{n\to\infty}\frac{1}{n} \log \nu ((J_{\mu_0}T_\bE^n)^{a})
\]
exists and has the value $\log \lambda_{\bE, a}$ by precisely the same calculation as above.  Thus
the moment generating function can be defined using $\nu$ in place of the invariant
measure $\mu_{\bE}$.  Note that since $c_{\bE, 0}(\nu) = 1$ for any probability measure $\nu \in \B$,
and due to the inequality
\[
\| (\Pi_{\bE, 0} - \Pi_{\bE, a}) \nu \|_{\B} \le \| \Pi_{\bE, 0} - \Pi_{\bE, a} \|_{\B} \| \nu \|_{\B}, 
\]
Lemma~\ref{lem:analytic} implies that if we fix a ball of radius $r>0$ in $\B$, then
we may choose $\ve_1$ so that $c_{\bE, a}(\nu) >0$ for all $\nu$ in this ball of radius $r$,
all $a \in [-a_0, 1+a_0]$ and all
$T_{\bE} \in \F(\ve_1, \tau_*, C_0)$.  For this range of parameters, it follows that Lebesgue measure $m$,
the smooth measure $\mu_0$ and the (possibly singular) SRB measure $\mu_{\bE}$
all yield the same logarithmic moment generating function $e_{\bE}(a)$.
From this and Proposition~\ref{ess} we conclude the symmetry $e_{\bE}(a) = e_{\bE}(1-a)$
for $a \in [-a_0, 1+a_0]$.

To prove item (3), we compute the derivatives of $e_\bE(a)$ at $a=0$,
following \cite{rey} (see also \cite{demers}).
The sequence
$\{\frac{1}{n} \log \mu_\bE((J_{\mu_0}T_\bE^n)^a)\}_{n \in \N}$
is uniformly bounded for $a$ in a complex neighborhood of the origin. Thus
by the Vitali
convergence theorem we can freely exchange derivative and limits. Thus
$$
e'_\bE(0)=\lim_{n\to\infty}\frac{1}{n} \mu_\bE (\log J_{\mu_0 }T_\bE^n)
=\mu_\bE(\log J_{\mu_0}T_\bE),
$$
due to the invariance of $\mu_{\bE}$ with respect to $T_{\bE}$. 
 Now using Lemma~\ref{TFJ}, we have $J_{\mu_0} T_\bE(x)=1+\eps H(x)$.
 Thus for small $|\eps|<1$,
 $$
 e'_\bE(0)=\mu_\bE(\log
(1+\eps H))=\eps \mu_\bE(H)+\cO(\eps^2).
$$
Next, using Lemma~\ref{lem:cont ve} and \cite[Corollary 1]{keller liverani}, 
we have 
\begin{equation}
\label{eq:meas close}
| \mu_{\bE} - \mu_0|_w \le C \ve^\eta,
\end{equation}
for some $\eta > 0$.  Putting these estimates together, we conclude,
\[
e'_{\bE}(0) = \ve \mu_0(H) + o(\ve).
\]

 For the second derivative, setting $s = -\log J_{\mu_0}T_\bE$ as before, and
$\bar s = s - \mu_\bE(s)$, we have
\begin{align*}
e''_\bE(0)&=\lim_{n\to\infty}\frac{1}{n}(\mu_\bE((S_n
s)^2)-\mu_\bE(S_ns)^2)=\lim_{n\to\infty}\frac{1}{n}\mu_\bE((S_n \bar s)^2)\\
&=\mu_\bE(\bar s^2)+2\lim_{n\to\infty}\sum_{j=1}^{n-1} (1-j/n) \mu_\bE(\bar s\cdot
\bar s\circ T_\bE^j)\\
&=\mu_\bE(\bar s^2)+2 \sum_{j=1}^{\infty}\mu_\bE(\bar s\cdot \bar s\circ T_\bE^j) .
\end{align*}
The last equality follows from the exponential decay of correlations and
dominated convergence.

Let $\sigma_{\bE}^2$ denote the limit of the variance of $n^{-1/2} S_ns$ as
$n \to \infty$ where $\{s \circ T_\bE^j\}_{j\in \mathbb{N}}$ is distributed
according to the invariant measure
$\mu_\bE$.
(Such a $\sigma_{\bE}$ exists and is finite
whenever the auto-correlations 
$\mu_\bE(\bar s\cdot \bar s\circ T_\bE^j)$ are summable). The above Green Kubo formula then
gives the diffusion coefficient:
\begin{equation}
\label{eq:GK}
e''_\bE(0)=\mu_\bE(\bar s^2)+2 \sum_{j=1}^{\infty}\mu_\bE(\bar s\cdot \bar s\circ T_{\bE}^j)
=\sigma_{\bE}^2 .
\end{equation}
We denote $\bar H=H-\mu_\bE(H)$, where $H$ is from
Lemma~\ref{TFJ} and $J_{\mu_0}T_\bE=1+\eps H$.  Then,
$$- \bar s=\log J_{\mu_0} T_{\bE}-\mu_\bE(\log J_{\mu_0} T_{\bE})=\eps (H-\mu_\bE(H))+\cO(\eps^2),$$
and also,
\begin{align*}
\mu_\bE(\bar s^2)&=\mu_\bE((\log(1+\eps H))^2)-\mu_\bE(\log(1+\eps H))^2\\
&=\eps^2\text{Var}(H)+\cO(\eps^3) = \eps^2 \mu_{\bE}(\bar H^2) + \cO(\eps^3).
\end{align*}
It follows that
\begin{align*}
\sum_{j=1}^{\infty}\mu_\bE(\bar s\cdot \bar s\circ T_\bE^j)
&=\sum_{j=1}^{\infty}\mu_\bE(\bar s\cdot \bar s\circ T_{\mathbf{0}}^j)+\sum_{j=1}^{\infty}\left(\mu_\bE(\bar s\cdot \bar s\circ T_\bE^j)-\mu_\bE(\bar s\cdot \bar s\circ T_{\mathbf{0}}^j)\right)\\
&=\eps^2 \sum_{j=1}^{\infty}\mu_\bE(\bar H\cdot \bar H\circ T_{\mathbf{0}}^j)+o(\eps^2).
\end{align*}
By exponential decay of correlations, the series in the last expression converges. 
Finally, we use \eqref{eq:meas close} to change the measure from $\mu_{\bE}$ to $\mu_0$
since all the functions involved are admissible with respect to the norms we have defined.
Thus 
$$e''_\bE(0)=\eps^2\sigma_H^2+o(\eps^2),$$
where $\sigma_H^2=\mu_0(\bar H^2)+2 \sum_{j=1}^{\infty}\mu_0(\bar H\cdot \bar H\circ T_{\mathbf{0}}^j)$.

Next, we show that in fact  $e_\bE(a)$ is strictly convex for $a \in [-a_0, 1+a_0]$ whenever
$\sigma_{\bE}^2>0$.

In order to compute $e'_\bE(a)$ and $e''_\bE(a)$ at $a\neq 0$, let
$$e_a(t):=\lim_{n\to\infty}\frac{1}{n}\log \mu_{\bE, a}(e^{-t S_n s}),$$ 
i.e. $e_a$
is the moment generating function for $\mu_{\bE,a}$. Note that
\begin{align*}
\mu_{\bE, a}(e^{-tS_n s})&=\langle e^{-tS_n s} \cdot h_{\bE, a} , \nu_{\bE, a} \rangle
=\lambda_{\bE, a}^{-n}\langle e^{-tS_n s} \Lp_a^n h_{\bE, a}, \nu_{\bE, a} \rangle 
=\lambda_{\bE, a}^{-n}\langle \Lp_{t+a}^nh_{\bE, a}, \nu_{\bE, a} \rangle.
\end{align*}
Therefore,
\begin{align*}
e_a(t)&=\lim_{n\to\infty}\frac{1}{n}\log \lambda_{\bE, a}^{-n}\langle \Lp^n_{t+a}h_{\bE, a}, \nu_{\bE, a} \rangle
=\lim_{n\to\infty}\frac{1}{n}\log \langle \Lp^n_{t+a}h_{\bE, a}, \nu_{\bE, a} \rangle - \log \lambda_{\bE,a} \\
&=\lim_{n\to\infty}\frac{1}{n}\log \langle \lambda_{\bE, a+t}^{n}c_{\bE, a+t}(h_{\bE, a})
h_{\bE, a+t}+R^n_{\bE, a+t} (h_{\bE, a}), \nu_{\bE, a} \rangle- \log \lambda_{\bE, a}
=e_\bE(a+t)-e_\bE(a),
\end{align*}
where we have used the fact that $e_\bE(a)=\log \lambda_{\bE, a}$. 
Differentiating with respect to $t$
gives $e'_\bE(a)=e'_a(0)$ and $e''_\bE(a)=e''_a(0)$. The computation of $e'_a(0)$
and $e_a''(0)$ are the same as the case $e'_\bE(0)$ and $e''_\bE(0)$, with $\mu_{\bE, a}$
in place of $\mu_{\bE}$.  

Notice that $\Lp_{T_{\mathbf{0}}, a} = \Lp_{T_{\mathbf{0}}, 0}$ for each $a \in \mathbb{R}$ since
when $\bE = (0,0)$, $s = 0$.  It follows that $\mu_{\mathbf{0}, a} = \mu_0$ for all $a \in \mathbb{R}$.  Thus by the continuity of $\mu_{\bE, a}$ in $\bE$ for each fixed $a$  (Lemma~\ref{lem:cont ve}),
we have $e''_a(0) >0$ for all $a \in [-a_0, 1+a_0]$ and $\ve < \ve_1$ if $\ve_1$ is chosen
sufficiently small.

The positivity of the entropy production rate follows then from the symmetry and 
strict convexity.  Indeed, suppose the entropy production rate $-e_{\bf E}'(0)=0$.  
Then since 
$e_{\bf E}(0)= e_{\bf E}(1)=0$ by the symmetry proved in item (2), convexity and analyticity imply that 
$e_{\bf E}(a) =0$ for all $a \in [0,1]$.  This
contradicts strict convexity, i.e. it contradicts that $e_{\bE}''(a)>0$ for all $a \in [0,1]$. 

It remains to prove that $\sigma_{\bE}^2 > 0$ if and only if $s$ is not a coboundary.
If $s = \psi \circ T_\bE- \psi +C$ for some $\psi$ then $e_\bE(a)=aC$ and so trivially
$e_{\bE}''(a) = 0$ for all $a$.  The converse requires a more substantial proof.  In order to prove it,
we will invoke the following abstract version of the Central Limit Theorem for invertible systems,
following the classical martingale approach of Gordin~\cite{gordin}.

\begin{theorem}[\cite{viana}]
\label{thm:CLT}
Let $(X, \cA, \mu)$ be a probability space, $\phi \in L^2(\mu)$ be such that $\int \phi \, d\mu = 0$,
and $\theta : X \to X$ be an invertible map such that both $\theta$ and $\theta^{-1}$ are measurable,
and $\mu$ is $\theta$-invariant and ergodic.  Let $\cA_0 \subset \cA$ be such that
$\cA_n = \theta^{-n}(\cA_0)$, $n \in \mathbb{Z}$, is a non-increasing sequence of $\sigma$-algebras.
Assume that 
\begin{equation}
\label{eq:conditions}
\sum_{n=0}^\infty \| E(\phi | \cA_n) \|_{L^2(\mu)} < \infty \quad \mbox{and} \quad
\sum_{n=0}^\infty \| \phi - E(\phi | \cA_{-n} ) \|_{L^2(\mu)} < \infty,
\end{equation}
and let $\sigma^2 = \int \phi^2 \, d\mu + 2 \sum_{j=1}^\infty \int \phi \cdot \phi \circ \theta^j \, d\mu$.

Then $\sigma$ is finite, and $\sigma = 0$ if and only if $\phi = u \circ \theta - u$ for some
$u \in L^2(\mu)$.  Moreover, if $\sigma^2 > 0$, then $n^{-1/2} S_n \phi$ converges in
probability to $\cN(0, \sigma^2)$. 
\end{theorem}

We will apply this theorem with $\theta = T_{\bE}$, $\mu = \mu_{\bE}$, $\phi = \bar s$
and $\sigma = \sigma_{\bE}$.  Let $\cA_0$ be the sigma-algebra generated by the
($\mu_{\bE}$-mod 0) partition of $M$ into 
maximal homogeneous local stable manifolds for $T_{\bE}$.\footnote{This partition is measurable
since it is has a countable generator:  $\cup_{n \ge 1} \{ \mbox{connected components of 
$M \setminus \cS_{n}^{T_{\bE}, \mathbb{H}}$}\}$.  See for example, \cite[Section 5.1]{chernov book}.}
Then $\cA_n = T_{\bE}^{-n}(\cA_0)$ is a decreasing sequence of sigma-algebras, as required.

With these definitions, the second condition in \eqref{eq:conditions} is a simple consequence of
the uniform contraction of stable manifolds.  Denoting by $V^s(x)$ the maximal local stable manifold
of $x$, we note that $E(\bar g| \cA_{-n})$ is constant on curves of the form $T_{\bE}^n(V^s(x))$;
these are the elements of $\cA_{-n}$ and their length is bounded by $C \Lambda^{-n}$ for
some uniform $C>0$
since stable manifolds have length
uniformly bounded above due to the discontinuities of $T_{\bE}$.
In fact, since $\bar s$ is continuous on such curves,\footnote{Indeed, $\bar s$ is H\"older continuous on each connected component of $M \setminus \cS_1^{T_{\bE}}$, and local stable manifolds cannot cross
$\cS_1^{T_{\bE}}$, otherwise they would be cut in forward time under $T_{\bE}$, which 
would contradict
the definition of stable manifold.}
$E(\bar s| \cA_{-n})(x) = \bar s(y)$ 
for some $y \in \cA_{-n}(x)$, the element of $\cA_{-n}$ containing $x$.  Thus
\[
\begin{split}
|\bar s(x) - E(\bar s | \cA_{-n})(x)| & = |\bar s(x) - \bar s(y)| = | \log (1+ \ve H(x)) - \log (1 + \ve H(y))| \\
& \le \frac{\ve}{1-C_H \ve} C_s^{\alpha_1}(H) d(x,y)^{\alpha_1}
\le C' \Lambda^{-n \alpha_1},
\end{split}
\]
where we have used Lemma~\ref{TFJ} and $C_s^{\alpha_1}( \cdot )$ denotes the H\"older constant
along stable manifolds with exponent $\alpha_1$.  This estimate
implies that the $L^\infty$-norm, and therefore the $L^2$-norm, of $\bar s(x) - E(\bar s | \cA_{-n})$ decays at an exponential rate
and so the second sum in \eqref{eq:conditions} converges.

The first sum in \eqref{eq:conditions} entails a more subtle calculation.  In principle, it follows
from exponential decay of correlations for $\mu_{\bE}$;  however, it requires exponential decay
of correlations against observables in $L^2(\cA_0, \mu_{\bE})$, which is a larger class than
is at first available in the framework of our Banach spaces; here, $L^2(\cA_0, \mu_{\bE})$
is the set of $L^2$ functions that are measurable with respect to $\cA_0$.  
To see this, we will use the dual version
of the $L^2$-norm,
\begin{equation}
\label{eq:dual}
\begin{split}
\| E( \bar s | \cA_n ) \|_{L^2(\mu_{\bE})} & = \sup \left\{ \int \bar s \, \phi \, d\mu_{\bE} : \phi \in L^2(\cA_n, \mu_{\bE}) \mbox{ with } \| \phi \|_{L^2(\mu_{\bE})} = 1 \right\} \\
& = \sup \left\{ \int \bar s \, \psi \circ T_{\bE}^n \, d\mu_{\bE} : \psi \in L^2(\cA_0, \mu_{\bE}) \mbox{ with }
\| \psi \|_{L^2(\mu_{\bE})} = 1 \right\} .
\end{split}
\end{equation}
In order for this last integral to decay exponentially in $n$ as a result of the spectral gap
for $\Lp_{T_{\bE}} = \Lp_{T_{\bE}, 0}$, we would like $\bar s \in \cB$ and
$\psi \in \cB'$, the dual to $\cB$.  Unfortunately, the first statement is false and the second statement
needs some work to justify.   Also, note that we can expect the correlations to decay to 0 in the
above expression since
$\mu_{\bE}(\bar s) = 0$.  It is not necessary that $\mu_{\bE}(\psi )= 0$ as well.

As noted in the proof of Lemma~\ref{lem:analytic}, as an admissible potential
$\bar s$ is H\"older continuous on connected components of $M \setminus \cS_1^{T_{\bE}}$
and so does not satisfy the assumptions of Lemma~\ref{lem:multiplier};  
however, $\bar s \circ T_{\bE}^{-1}$ is $\alpha_1/2$-H\"older continuous on connected 
components of $M \setminus \cS_{-1}^{T_{\bE}}$ by \eqref{eq:holder}.
Thus by Lemma~\ref{lem:multiplier}, since $\gamma < 1/7 \le \alpha_1/(\alpha_1 + 2)$ in the
definition of the norms, both $\bar s \circ T_{\bE}^{-1}$ and $\bar s \circ T_{\bE}^{-1} h_{\bE} \in \cB$, 
where $h_{\bE}$ is the right eigenvector of $\Lp_{T_{\bE}}$.
Since $\int \bar s \, \psi \circ T_{\bE}^n \, d\mu_{\bE} = \int \bar s \circ T_{\bE}^{-1} \, \psi \circ T_{\bE}^{n-1} \, d\mu_{\bE}$, it suffices to work with $\bar s \circ T_{\bE}^{-1}$.

Notice also that since we are in the case $a=0$, the conformal measure is $\mu_0$, i.e.
$\Lp_{T_{\bE}}^* \mu_0 = \mu_0$.  Thus for $n \ge 0$ and $\psi \in \C^\alpha(T_{\bE}^{-n}\cW^s)$,
we have $\mu_{\bE} = h_{\bE} \otimes \mu_0$ and
$\mu_{\bE}(\psi) = \langle h_{\bE}, \psi \mu_0 \rangle$.

In order to estimate the expression in \eqref{eq:dual}, we shall need two lemmas.  Let $B_0(\cA_0)$
denote the set of bounded functions on $M$, which are measurable with respect to $\cA_0$.

\begin{lemma}
\label{lem:decay}
Suppose there exist $C>0$ (depending on $\bar g$) and $\rho < 1$ such that
\begin{equation}
\label{eq:inf decay}
\mu_{\bE} (\bar s \cdot \psi \circ T_{\bE}^n ) \le C \rho^n |\psi|_\infty, \quad
\mbox{for all $\psi \in B_0(\cA_0)$,}
\end{equation}
where $|\psi|_{\infty} = \sup_{x \in M} |\psi(x)|$. 
Then there exists $C'>0$ such that
\[
\mu_{\bE}(\bar s \cdot \psi \circ T_{\bE}^n) \le C' \rho^{n/2} |\psi|^2_{L^2(\mu_{\bE})},
\quad \mbox{for all $\psi \in L^2(\cA_0, \mu_{\bE})$.}
\]
\end{lemma}

The following lemma is a strengthening of Lemma~\ref{lem:distr}(i).  It shows that the estimate
of that lemma holds true in the limit as $n \to \infty$.

\begin{lemma}
\label{lem:dual}
There exists $C>0$ such that for any $h \in \B_w$ and any bounded function $\psi$,
\[
|h(\psi)| \le C |h|_w (|\psi|_{\infty} + C^\alpha_{\cA_0}(\psi)),
\]
where $C^\alpha_{\cA_0}(\cdot)$ denotes the H\"older constant of $\psi$ with exponent $\alpha$
measured along curves in $\cA_0$.
\end{lemma}

We postpone the proofs of the lemmas and first show how they allow us to complete the proof
of Theorem~\ref{thm:moment}. 
For $\psi \in \cB_0(\cA_0)$ , $C^\alpha_{\cA_0}(\psi) =0$ since $\psi$ is constant on curves in $\cA_0$.
We estimate the correlations using Lemma~\ref{lem:dual},
\begin{equation}
\label{eq:decay}
\begin{split}
\left| \int \bar s \, \psi \circ T_{\bE}^n \, d\mu_0 \right| 
& = \left| \int \bar s \circ T_{\bE}^{-1} \, \psi \circ T_{\bE}^{n-1} \, d\mu_{\bE} \right|
 = \left| \langle \bar s \circ T_{\bE}^{-1} \, h_{\bE}, \psi \circ T_{\bE}^{n-1} \, \mu_0 \rangle \right| \\
& = \left| \langle \bar s \circ T_{\bE}^{-1} \, h_{\bE}, (\Lp_{T_{\bE}}^*)^{n-1} (\psi \mu_0) \rangle \right| 
 = \left| \langle \Lp_{T_{\bE}}^{n-1}(\bar s \circ T_{\bE}^{-1} \, h_{\bE}), \psi \mu_0 \rangle \right| \\
& \le C | \Lp_{T_{\bE}}^{n-1}( \bar s \circ T_{\bE}^{-1} h_{\bE})|_w  ( |\psi|_\infty + H^\alpha_{\cA_0}(\psi)) 
\le C \| R_{\bE}^{n-1}(\bar s \circ T_{\bE}^{-1} h_{\bE})\|_{\B} |\psi|_{\infty} \\
& \le C \rho^n \| h_{\bE} \|_{\cB} | \bar s \circ T_{\bE}^{-1} |_{\C^{\alpha_1/2}(\cP_{-1})} |\psi|_{\infty} ,
\end{split}
\end{equation}
for some $\rho<1$
where $| \bar s \circ T_{\bE}^{-1} |_{\C^{\alpha_1/2}(\cP_{-1})}$ denotes the H\"older constant of
$\bar s$ on elements of the partition formed by the connected components of $M \setminus \cS_{-1}^{T_{\bE}}$, and we have used \eqref{eq:decomp} and
the fact that $\Pi_{\bE} (\bar s \circ T_{\bE}^{-1} h_{\bE}) = 0$
since $\langle \bar s \circ T_{\bE}^{-1} h_{\bE}, \mu_0 \rangle 
= \mu_{\bE}(\bar s \circ T_{\bE}^{-1}) = \mu_{\bE}(\bar s) = 0$.

From \eqref{eq:decay}, we see that $\bar s$ has uniform exponential
decay of correlations against $\psi \in B_0(\cA_0)$ and so satisfies the hypotheses of 
Lemma~\ref{lem:decay}.  It follows that $\bar s$ enjoys a
uniform exponential rate of decay of correlations against $\psi \in L^2(\cA_0, \mu_{\bE})$, so
by \eqref{eq:dual}, this yields an exponential decay in the $L^2$-norm of $E(\bar s| \cA_n)$.
We conclude that the first series in \eqref{eq:conditions} converges.  
Since the hypotheses of Theorem~\ref{thm:CLT} are verified,
it follows that $\sigma_{\bE}^2 = 0$ if and only if $\bar s = u \circ T_{\bE} - u$ for some
$u \in L^2(\mu_{\bE})$, and the proof of Theorem~\ref{thm:moment} is complete.
\end{proof}


\begin{proof}[Proof of Lemma~\ref{lem:decay}]
Let $\psi \in L^2(\cA_0, \mu_{\bE})$ be arbitrary.  For $L \in \mathbb{R}^+$, define
$\psi_L(x) = \psi(x)$ when $|\psi(x)| \le L$ and $\psi_L(x) = 0$ otherwise.  Clearly, 
$\psi_L \in B_0(\cA_0)$ and $|\psi_L|_\infty \le L$.  Now for $n \in \mathbb{N}$,
\begin{equation}
\label{eq:split}
\left| \int \bar s \cdot \psi \circ T_{\bE}^n \, d\mu_{\bE} \right| \le \left| \int \bar s \cdot \psi_L \circ T_{\bE}^n \, d\mu_{\bE} \right| + | \bar s|_{\infty} \int |\psi - \psi_L| \, d\mu_{\bE}  .
\end{equation}
To bound the second term on the right side of \eqref{eq:split}, note that
\[
\int |\psi - \psi_L | \, d\mu_{\bE} \le \int 1_{|\psi| > L} \cdot |\psi| \, d\mu_{\bE}
\le \mu_{\bE}(|\psi|>L)^{1/2} | \psi |_{L^2(\mu_{\bE})},
\]
while $\mu(|\psi| > L) = \mu(\psi^2 > L^2) \le L^{-2} | \psi |^2_{L^2}$, by Markov's inequality.
Using \eqref{eq:inf decay} for the first term on the right side of \eqref{eq:split}, we obtain,
\begin{equation}
\label{eq:uni}
\left| \int \bar s \cdot \psi \circ T_{\bE}^n \, d\mu_{\bE} \right| \le C \rho^n L + L^{-1} |\bar s|_{\infty} | \psi |^2_{L^2(\mu_{\bE})} \le \rho^{n/2} (C + |\bar s|_{\infty} | \psi |^2_{L^2(\mu_{\bE})}),
\end{equation}
if we set $L = \rho^{-n/2}$, and the lemma is proved.
\end{proof}

\begin{proof}[Proof of Lemma~\ref{lem:dual}]
Due to the density of $\C^1(M)$ in $\cB_w$, it suffices to prove the lemma for $h \in \C^1(M)$.

On each component $M_i$ of $M$, $i = 1, \ldots, d$, 
we disintegrate the smooth measure $\mu_0$ on elements
of $\cA_0$.  Since elements of $\cA_0$ are homogeneous stable manifolds, the decomposition
respects the boundaries of the homogeneity strips.  Let $\cA_{0, i} = \{ W_\xi \}_{\xi \in \Xi_i}$ 
denote the set of homogeneous local stable manifolds in $M_i$ with index set $\Xi_i$.  The disintegration
of $\mu_0$ on elements of $\cA_{0,i}$ yields conditional densities $\eta_{\xi}$ on $W_\xi$, normalized
so that $\int_{W_\xi} \eta_\xi \, dm_{W_\xi} = 1$, and a factor measure $\hmu_0$ on the index set
$\Xi_i$.  By \cite[Corollary~5.30]{chernov book}, $\eta_\xi$ is (uniformly in $\xi$) log-H\"older
continuous with exponent $1/3$.  Now,
\begin{equation}
\label{eq:distr}
h(\psi) = \int_M h \psi \, d\mu_0 = \sum_i \int_{\Xi_i} \int_{W_\xi} h \,\psi \,\eta_\xi \, dm_{W_\xi} d\hmu_0(\xi) .
\end{equation}
On each $W_\xi$, we estimate using the weak norm.
\[
\left| \int_{W_\xi} h \,\psi \,\eta_\xi \, dm_{\xi} \right|
\le |h|_w |\psi|_{\C^\alpha(W_\xi)} |\eta_{\xi}|_{\C^\alpha(W_\xi)} .
\]
Due to the log-H\"older regularity of $\eta_\xi$, there exists a constant $C_\eta > 0$ such that
\[
| \eta_\xi(x) - \eta_\xi(y) | \le C_\eta |\eta_\xi(x)| d(x,y)^{1/3} \le C' |W_\xi|^{-1} d(x,y)^{1/3},
\]
where the bound on the sup-norm of $\eta_\xi$ comes from the normalization of the conditional measures.
Putting these estimates into \eqref{eq:distr} we obtain,
\[
| h(\psi) | \le C' |h|_w (|\psi|_\infty + H^\alpha_{\cA_0}(\psi)) \sum_i \int_{\Xi_i} |W_\xi|^{-1} \, d\hmu_0(\xi) .
\]
This last integral is precisely the integral that characterizes the $Z$-function for a standard family
\cite[Section~7.4]{chernov book} which measures the prevalence of short curves in that family.
Since the disintegration of $\mu_0$ on maximal homogeneous stable manifolds creates a proper
family,\footnote{In fact, Example 7.21 and Exercise 7.22 of \cite{chernov book} are stated in terms
of the disintegration of $\mu_0$ on maximal homogeneous unstable manifolds.  Using the
reversibility of $T_{\bE}$, the analogous properties hold for maximal homogeneous stable manifolds.} 
this integral is finite (see \cite[Exercise~7.22]{chernov book} for the decomposition using stable 
manifolds for the unperturbed billiard $T_0$ and
\cite{CZ09, CZZ} for the decomposition using stable manifolds for the perturbed billiard $T_{\bE}$). 
\end{proof}

\appendix

\section{Lasota-Yorke Estimates}
\label{appendix}

\subsection{Preliminary estimates}
\label{preliminary}

Before proving the Lasota-Yorke inequalities, we show how {\bf (H1)}-{\bf
(H5)} imply
several other uniform properties for our class of maps $\F$. In
particular, we will be
interested in iterating the one-step expansion given by {\bf
(H3)}.
We recall the estimates we need from \cite[Section 3.2]{demers zhang}.

Let $T \in \F$ and $W \in \W^s$. Let $V_i$ denote the maximal connected
components
of $T^{-1}W$ after cutting due to singularities and the boundaries of the
homogeneity
strips. To ensure that each component of $T^{-1}W$ is in $\W^s$, we
subdivide any of the
long pieces $V_i$ whose length is $>\delta_0$, where $\delta_0$ is chosen
in \eqref{eq:one step contract}.        This process is then iterated
so that
given $W \in \W^s$, we construct the components of $T^{-n}W$, which we
call
the $n^{\mbox{\scriptsize th}}$ generation $\G_n(W)$, inductively as
follows.
Let $\G_0(W) = \{ W \}$ and suppose we have
defined $\G_{n-1}(W) \subset \W^s$. First, for any $W' \in \G_{n-1}(W)$,
we partition $T^{-1}W'$ into at most countably many pieces $W'_i$ so that
$T$ is
smooth on each $W'_i$ and
each $W'_i$ is a homogeneous stable curve.
If any $W'_i$ have length greater than $\delta_0$, we subdivide those
pieces into pieces
of length between $\delta_0/2$ and $\delta_0$.
We define $\G_n(W)$ to be the collection of all pieces $W^n_i \subset
T^{-n}W$ obtained
in this way.  Note that each $W^n_i$ is in $\W^s$ by {\bf (H2)}.

At each iterate of $T^{-1}$, typical curves in $\G_n(W)$ grow in size, but
there exist a portion of curves which are trapped in tiny homogeneity
strips and in the infinite horizon case, stay too close to the infinite
horizon points. In Lemma~\ref{lem:growth}, we make precise the sense in
which the proportion
of curves that never grow to a fixed
length decays exponentially fast.

For $W \in \W^s$, $n \geq 0$, and $0 \le k \le n$, let $\G_k(W) = \{ W^k_i
\}$ denote
the $k^{\mbox{\scriptsize th}}$ generation pieces in $T^{-k}W$.  Let
$B_k(W) = \{ i : |W^k_i|< \delta_0/3 \}$ and $L_k(W) = \{ i : |W^k_i| \ge
\delta_0/3 \}$
denote the index of the short and long elements of $\G_k(W)$,
respectively.
We consider
$\{\G_k \}_{k=0}^n$ as a tree with $W$ as its root and $\G_k$ as the
$k^{\mbox{\scriptsize th}}$ level.

At level $n$, we group the pieces as follows. Let $W^n_{i_0} \in \G_n(W)$
and
let $W^k_j \in L_k(W)$ denote the most recent long ``ancestor" of
$W^n_{i_0}$, i.e.\
$k = \max \{ 0 \leq \ell \le n : T^{n-\ell}(W^n_{i_0}) \subset W^\ell_j \;
\mbox{and} \;
j \in L_\ell \}$. If no such ancestor exists, set $k=0$ and $W^k_j = W$.
Note that
if $W^n_{i_0}$ is long, then $W^k_j = W^n_{i_0}$.  Let
\[
\I_n(W^k_j) = \{ i : W^k_j \in L_k(W) \; \mbox{is the most recent long
ancestor of} \; W^n_i \in \G_n(W) \}.
\]
The set $\I_n(W)$ represents those curves $W^n_i$ that belong to short
pieces in $\G_k(W)$
at each time step $1 \leq k \le n$, i.e.\ such $W^n_i$ are never part of a
piece
that has grown to length $\geq \delta_0/3$.

We collect the necessary complexity estimates in the
following lemma.

\begin{lemma}
\label{lem:growth} 
Let $W \in \W^s$, $T \in \F$ and for $n \geq 0$, let $\I_n(W)$ and
$\G_n(W)$
be defined as above.
There exist constants $C_1, C_2, C_3>0$, independent of $W$ and $T$, such
that for
any $n\geq 0$,
\begin{itemize}
  \item[(a)] $\ds
\sum_{i \in \I_n(W)} |J_{W^n_i}T^n|_{\C^0(W^n_i)} \leq C_1 \theta_*^n $;
 \item[(b)] $\ds
\sum_{W^n_i \in \G_n(W)} |J_{W^n_i}T^n|_{\C^0(W^n_i)} \le C_2     $;
 \item[(c)]  for any $0 \leq \varsigma \leq 1$,
$\ds
\sum_{W^n_i \in \G_n(W)} \frac{|W^n_i|^\varsigma}{|W|^\varsigma} \;
|J_{W^n_i}T^n|_{\C^0(W^n_i)} \le C_2^{1-\varsigma} $.
 \item[(d)] for $\varsigma > 1/2$,
$\ds
\sum_{W^n_i \in \G_n(W)} |J_{W^n_i}T^n|_{\C^0(W^n_i)}^\varsigma \le
C_3^n$,
where $C_3$ depends on $\varsigma$.
\end{itemize}
\end{lemma}

\begin{proof}
Item (a) follows from the one-step expansion {\bf (H3)} by induction as in \cite[Lemma 3.1]{demers zhang}.  Items (b) and (c) are precisley \cite[Lemmas 3.2 and 3.3]{demers zhang}.

For item (d), we first prove that the claimed estimate holds for $n=1$.  Indeed, due to {\bf (H1)},
the expansion for each stable curve landing in a homogeneity strip $\Ho_k$ under $T^{-1}$ is of the
order of $k^{-2}$.  If $T^{-1}$ crosses a countable number of singularity curves, the sum of the
expansion factors is uniformly bounded as long as $\varsigma > 1/2$.  Since there are only
finitely many genuine singularity curves in $\cS_{-1}^T$ (not counting homogeneity strips),
the required sum is uniformly bounded for all $W \in \cW^s$ with $|W| \le \delta_0$.  The estimate
for general $n$ follows by induction as in \cite[Lemma 3.4]{demers zhang}.
\end{proof}

Next we state a distortion bound for the stable Jacobian of $T$ along
different stable curves in the following context.
Let $W^1, W^2 \in \W^s$ and suppose there exist $U^k \subset T^{-n}W^k$,
$k=1,2$,
such that for $0 \le i \le n$,
\begin{enumerate}
\item[(i)] $T^iU^k \in \W^s$ and the curves $T^iU^1$ and $T^iU^2$ lie in
the same homogeneity
strip;
  \item[(ii)]  $U^1$ and $U^2$ can be put into a 1-1 correspondence
by a smooth foliation $\{ \gamma_x \}_{x \in U^1}$ of curves
$\gamma_x \in \widehat\W^u$ such that
$\{ T^n\gamma_x \} \subset \widehat\W^u$ creates a 1-1 correspondence
between
$T^nU^1$ and $T^nU^2$;
\item[(iii)] $|T^i\gamma_x| \le 2 \max \{ |T^iU^1|, |T^iU^2| \}$, for all
$x \in U^1$.
\end{enumerate}
Let $J_{U^k}T^n$ denote the stable Jacobian of $T^n$ along the
curve $U^k$ with respect to arclength. The following lemma was proved in
\cite{DZ2}.

\begin{lemma}
\label{lem:angles}
In the setting above, for $x \in U^1$, define $x^* \in \gamma_x \cap U^2$.
There
exists $C_0 > 0$, independent of $T \in \F$, $W \in \W^s$ and $n \ge 0$
such that
\begin{enumerate}
\item[(a)]  $d_{\W^s}(U^1, U^2) \le C_0 \Lambda^{-n} d_{\W^s}(W^1, W^2)$;  \item[(b)] $\ds
\left| \frac{J_{U^1}T^n(x)}{J_{U^2}T^n(x^*)} -1 \right| \; \leq \;
C_0[d(T^nx,T^n x^*)^{1/3}
+\theta(T^nx, T^n x^*)] $,
\end{enumerate}
where $\theta(T^nx, T^n x^*) $ is the angle formed by the tangent lines of 
$T^nU^1$ and $T^nU_2$ at $T^nx$ and $T^n x^*$, respectively.
\end{lemma}


To prove Proposition~\ref{prop:ly}, we fix $T \in \F$ and prove the
required Lasota-Yorke inequalities
\eqref{eq:weak norm}-\eqref{eq:unstable norm}.
It is shown in Lemma \ref{lem:L well defined} that $\Lp_{T,g}$ is a
continuous operator
on both $\B$ and $\B_w$
so that it suffices to prove the inequalities for $h \in \C^1(M)$. They
extend to the
completions by continuity.
Our purpose now is to show how they depend explicitly on the uniform
constants
given by {\bf (H1)}-{\bf (H5)} and do not require additional information.


\subsection{Estimating the weak norm}
\label{weak norm}

Let $h \in \C^1(M)$, $W \in \W^s$ and $\psi \in \C^\alpha(W)$ such that
$|\psi|_{\C^\alpha(W)} \leq 1$.      For brevity, we define
\[
\hg = g - \log J_{\mu_0} T , \qquad \mbox{so that} \qquad e^{S_n \hg} = e^{S_n g} (J_{\mu_0}T^n)^{-1} .
\]   
For $n \geq0$, we write,
\begin{equation}
\label{eq:start}
\int_W\Lp_{T,g}^nh \, \psi \, dm_W
=\sum_{W^n_i \in\G_n(W)}\int_{W^n_i}h e^{S_n\hg} J_{W^n_i}T^n\psi
\circ T^n dm_W
\end{equation}
where $J_{W^n_i}T^n$ denotes the Jacobian of $T^n$ along $W^n_i$.

Using the definition of the weak norm on each $W^n_i$, we estimate
\eqref{eq:start} by
\begin{equation}
\label{eq:weak estimate}
\int_W\Lp_{T,g}^nh \, \psi \, dm_W \; \leq \; \sum_{W^n_i \in\G_n} |h|_w
|J_{W^n_i}T^n|_{\C^\alpha(W^n_i)} |e^{S_n
\hat g}|_{\C^\alpha(W^n_i)} |\psi\circ T^n|_{\C^\alpha(W^n_i)} .
\end{equation}
Using the bounded distortion property {\bf (H4)}, we estimate,
\beq
\label{eq:holder c0}
|J_{W^n_i} T^n|_{\C^\alpha(W^n_i)} \le (1+C_d) |J_{W^n_i} T^n|_{\C^0(W^n_i)},
\eeq
and similarly for $|(J_{\mu_0}T^n)^{-1}|_{\C^\alpha(W^n_i)}$.  Next, for the potential $g$ and
$x, y \in W^n_i$,
\beq
\label{eq:g dist}
|e^{S_ng(x)} - e^{S_ng(y)}| \le |e^{S_n g}|_{\C^0(W^n_i)} |S_ng(x) - S_ng(y)| \le
|e^{S_n g}|_{\C^0(W^n_i)} |g|_{\C^\alpha(M)} \sum_{i=0}^{n-1} C_e \Lambda^{-i\alpha} |x-y|^\alpha, 
\eeq
so that $|e^{S_ng}|_{\C^p(W^n_i)} \le C_g |e^{S_ng}|_{L^\infty}$, where
$C_g := C_e |g|_{\C^\alpha(M)} \sum_{i=0}^\infty \Lambda^{-i\beta}$.

Finally, we esimate the norm of $\psi \circ T^n$, again using {\bf (H1)}.
For $x,y \in W^n_i$, \begin{equation}
\label{eq:C1 C0}
\frac{|\psi (T^nx) - \psi (T^ny)|}{d_W(T^nx,T^ny)^\alpha}\cdot
\frac{d_W(T^nx,T^ny)^\alpha}{d_W(x,y)^\alpha} \leq |\psi|_{\C^\alpha(W)}
|J_{W^n_i}T^n|^\alpha_{\C^0(W^n_i)}
    \leq C_e \Lambda^{-\alpha n} |\psi|_{\C^\alpha(W)} ,
\end{equation}
so that $|\psi \circ T^n|_{\C^\alpha(W^n_i)} \leq C_e |\psi|_{\C^\alpha(W)} \le
C_e$.
We use this estimate together with \eqref{eq:holder c0} and \eqref{eq:g dist} to
bound
\eqref{eq:weak estimate} by
\[
\int_W \Lp_{T,g}^nh \, \psi \, dm_W \leq C_e (1+C_d)^2 C_g |e^{S_n \hg}|_{\infty} |h|_w \sum_{W^n_i \in \G_n}
|J_{W^n_i}T^n|_{\C^0(W^n_i)} \le C' C_g |e^{S_n \hg}|_{\infty} |h|_w,
\]
where $C' = C_e (1+C_d)^2 C_2$ and we have used Lemma~\ref{lem:growth}(b)
for the last inequality.
Taking the supremum over all $W \in \W^s$ and $\psi \in \C^\alpha(W)$ with
$|\psi|_{\C^\alpha(W)} \leq 1$
yields \eqref{eq:weak norm} expressed with uniform constants given by {\bf
(H1)}-{\bf (H5)}.


\subsection{Estimating the strong stable norm}
\label{stable norm}

Let $W \in \W^s$ and let $W^n_i$ denote the elements of $\G_n(W)$ as
defined above.
For $\psi \in \C^\beta(W)$, $|\psi|_{\C^\beta(W)} \le |W|^{-p}$, define
$\bpsi_i = |W^n_i|^{-1} \int_{W^n_i} \psi \circ T^n \, dm_W$.
Using equation \eqref{eq:start}, we write
\begin{equation}
\label{eq:stable split}
\int_W\Lp_{T,g}^nh\, \psi \, dm_W      =
\sum_{i} \int_{W^n_i} e^{S_n \hat g} h \cdot J_{W^n_i}T^n \cdot(\psi \circ
T^n- \bpsi_i) \, dm_W
 + \bpsi_i  \int_{W^n_i}h e^{S_n \hat g} \cdot J_{W^n_i}T^n \, dm_W .
\end{equation}

To estimate the first term of \eqref{eq:stable split},
we first estimate $|\psi \circ T^n - \bpsi_i|_{\C^\beta(W^n_i)}$.
If $H_W^\beta(\psi)$ denotes the H\"older constant of $\psi$ along $W$, then
equation~\eqref{eq:C1 C0} implies
\begin{equation}
\label{eq:H^beta}
\frac{|\psi (T^nx) - \psi (T^ny)|}{d_W(x,y)^\beta} \leq C_e \Lambda^{-n\beta}
H_W^\beta(\psi)
\end{equation}
for any $x,y \in W^n_i$.  Since $\bpsi_i$ is constant on $W^n_i$, we have
$H^\beta_{W^n_i}(\psi \circ T^n - \bpsi_i) \leq C_e \Lambda^{-\beta n} H^\beta_W(\psi)$.
To estimate the $\C^0$ norm, note that $\bpsi_i = \psi \circ T^n(y_i)$ for
some
$y_i \in W^n_i$.  Thus for each $x \in W^n_i$,
\[
| \psi \circ T^n(x) - \bpsi_i|
    = |\psi \circ T^n(x) - \psi \circ T^n(y_i)|
\leq H^\beta_{W^n_i}(\psi \circ T^n) |W^n_i|^\beta \leq C_e H^\beta_W(\psi)
\Lambda^{-\beta n} .
\]
This estimate together with \eqref{eq:H^beta} and the fact that
$|\vf|_{W,p, \beta} \leq 1$, implies
\begin{equation}
\label{eq:C^beta small}
|\psi \circ T^n - \bpsi_i|_{\C^\beta(W^n_i)} \leq C_e \Lambda^{- \beta n}
|\psi|_{\C^\beta(W)}
\leq C_e \Lambda^{- \beta n} |W|^{-p} .
\end{equation}

We apply \eqref{eq:holder c0}, \eqref{eq:g dist} and \eqref{eq:C^beta small}
and the definition of the strong stable norm to the first term of
\eqref{eq:stable split},
\begin{equation}
\label{eq:first stable}
\begin{split}
\sum_i \int_{W^n_i} h e^{S_n \hat g} & J_{W^n_i}T^n \, (\psi \circ T^n -
\bpsi_i) \, dm_W \leq
(1+C_d)^2 C_e \sum_i \|h\|_s \frac{|W^n_i|^p}{|W|^p}
\left|e^{S_n \hat g} J_{W^n_i}T^n \right|_{C^0(W^n_i)} \Lambda^{- \beta n} \\
& \leq \; |e^{S_n \hg}|_{\infty} C_g (1+C_d)^2 C_e C_g \Lambda^{- \beta n} \|h\|_s
\sum_i \frac{|W^n_i|^p}{|W|^p}
|J_{W^n_i}T^n|_{\C^0(W^n_i)}
\; \leq \; C_4 C_g | e^{S_n \hg}|_{\infty} \Lambda^{-\beta n} \|h\|_s ,
\end{split}
\end{equation}
where $C_4 = (1+C_d)^2 C_e C_2^{1-p}$ and
in the second line we have used 
Lemma~\ref{lem:growth}(c) with $\varsigma = p$.

For the second term of \eqref{eq:stable split}, we use the fact that
$|\bpsi_i| \leq |W|^{-p} $ since $|W|^p |\psi|_{\C^\beta(W)} \le 1$.
Recall the notation introduced before the statement of
Lemma~\ref{lem:growth}.
Grouping the pieces $W^n_i \in \G_n(W)$
according to most recent long ancestors $W^k_j \in L_k(W)$, we have
\[
\begin{split}
\sum_{i} |W|^{-p} \int_{W^n_i}h e^{S_n \hat g} \cdot J_{W^n_i}T^n \,
dm_W
= & \sum_{k=1}^n \sum_{j\in L_k(W)}\sum_{ i\in \I_n(W^k_j)}
|W|^{-p} \int_{W^n_i}h e^{S_n \hat g} \cdot J_{W^n_i}T^n \, dm_W \\&  + \sum_{ i\in \I_n(W)}
     |W|^{-p} \int_{W^n_i} h e^{S_n \hat g} J_{W^n_i}T^n \, dm_W
\end{split}
\]
where we have split up the terms involving $k=0$ and $k \geq 1$.
We estimate the terms with $k \geq 1$
by the weak norm and the terms with $k=0$ by the strong stable norm.
Using again \eqref{eq:holder c0} and \eqref{eq:g dist},
\[
\begin{split}
\sum_{i} |W|^{-p} \int_{W^n_i}he^{S_n \hat g} \cdot J_{W^n_i}T^n \,
dm_W
& \leq
|e^{S_n\hg}|_{\infty} C_g (1+C_d)^2 \sum_{k=1}^n\sum_{j\in L_k(W)}
\sum_{i\in \I_n(W^k_j)} |W|^{-p} |h|_w | J_{W^n_i}T^n
|_{\C^0(W^n_i)}\\
&  \; \; \; \; \; +
|e^{S_n\hg}|_{\infty} C_g (1+C_d)^2 \sum_{\ i\in \I_n(W)}
\frac{|W^n_i|^p }{|W|^p } \|h\|_s |J_{W^n_i}T^n|_{\C^0(W^n_i)} .
\end{split}
\]

In the first sum above corresponding to $k\geq 1$, we write
\[
|J_{W^n_i}T^n|_{\C^0(W^n_i)} \leq |J_{W^n_i}T^{n-k}|_{\C^0(W^n_i)}
|J_{W^k_j}T^k|_{\C^0(W^k_j)} .
\]
Thus using Lemma~\ref{lem:growth}(a) from time $k$ to time $n$,
\[
\begin{split}
\sum_{k=1}^n \sum_{j \in L_k} \sum_{i \in \I_n(W^k_j)} |W|^{-p}
|J_{W^n_i}T^n|_{\C^0(W^n_i)}
& \leq \sum_{k=1}^n \sum_{j \in L_k(W)} |J_{W^k_j}T^k|_{\C^0(W^k_j)}
|W|^{-p}
\sum_{i  \in \I_n(W^k_j)} |J_{W^n_i}T^{n-k}|_{\C^0(W^n_i)} \\
& \le 3 \delta_0^{-p} \sum_{k=1}^n \sum_{j \in L_k(W)}
|J_{W^k_j}T^k|_{\C^0(W^k_j)} \,
\frac{|W^k_j|^p}{|W|^p} C_1 \theta_*^{n-k},
\end{split}
\]
since $|W^k_j| \ge \delta_0/3$.
The inner sum is bounded by $C_2^{1-p}$ for each $k$ by
Lemma~\ref{lem:growth}(c)
while the outer sum is bounded by $C_1/(1-\theta_*)$ independently of $n$.

Finally, for the sum corresponding to $k=0$, since
\[
|J_{W^n_i}T^n|_{\C^0(W^n_i)} \le (1+C_d) |T^nW^n_i| |W^n_i|^{-1} \le
(1+C_d) |J_{W^n_i}T^n|_{\C^0(W^n_i)},
\]
we use Jensen's inequality and Lemma~\ref{lem:growth}(a) to estimate,
\[
\sum_{i \in \I_n(W)} \frac{|W^n_i|^p}{|W|^p}
|J_{W^n_i}T^n|_{\C^0(W^n_i)}
\le (1+C_d)
\left( \sum_{i \in \I_n(W)} \frac{|T^nW^n_i|}{|W^n_i|} \right)^{1-p}
\le (1+C_d)C_1 \theta_*^{n(1-p)}.
\]

Gathering these estimates together, we have
\begin{equation}
\label{eq:second stable}
\sum_{i} |W|^{-p}\left| \int_{W^n_i}h e^{S_n \hat g} J_{W^n_i}T^n \,
dm_W\right|
\; \leq \; C_g |e^{S_n\hg}|_{\infty} \Big( C_5 \delta_0^{-p}|h|_w + C_6
\|h\|_s  \theta_*^{n(1-p)} \Big),
\end{equation}
where $C_5 = 3(1+C_d)^2 C_1C_2^{1-p} /(1-\theta_*)$ and
$C_6 = (1+C_d)^3 C_1$.
Putting together \eqref{eq:first stable} and \eqref{eq:second stable}
proves
\eqref{eq:stable norm},
\[
\|\Lp_{T,g}^n h\|_s \leq C' C_g |e^{S_n\hg}|_\infty \left( \Lambda^{-\beta
n}+\theta_*^{n(1-p)}\right)\|h\|_s
         + C' C_g |e^{S_n\hg}|_\infty \delta_0^{-p} |h|_w ,
\]
with $C' = \max \{ C_4, C_5, C_6 \}$, a uniform constant depending only on {\bf (H1)}-{\bf (H5)}.


\subsection{Estimating the strong unstable norm}
\label{unstable norm}

Fix $\ve \le \ve_0$ and
consider two curves $W^1, W^2 \in\W^s$ with $d_{\W^s}(W^1,W^2) \leq \ve$.
For $n \geq 1$, we describe how to partition $T^{-n}W^\ell$ into
``matched'' pieces $U^\ell_j$ and
``unmatched'' pieces $V^\ell_k$, $\ell=1,2$.
In the what follows, we use $C_t$ to
denote a transversality constant which depends only on the minimum angle
between various transverse directions: the minimum angle between $C^s(x)$
and
$C^u(x)$, between $S^T_{-n}$ and $C^s(x)$, and
between $C^s(x)$ and the vertical and horizontal directions.

Let $\omega$ be a connected component of $W^1 \setminus \Si_{-n}^T$
such that $T^{-n}\omega \in \G_n(W)$.
We define a smooth local foliation $\{ \gamma_x \}_{x \in T^{-n}\omega}$
about $T^{-n}\omega$
such that for each $x \in T^{-n}\omega$: (1) $\gamma_x$ is centered at
$x$,
(2) $\gamma_x \in \widehat W^u$; (3) $|\gamma_x| \le 2BC_t C_e
\Lambda^{-n}\ve$
such that its image $T^n\gamma_x$, if not cut by a singularity or the
boundary of
a homogeneity strip, will have a projection on the vertical direction of
length $2\ve$.
By item (3) and the definition of $d_{\W^s}(W^1, W^2)$, it follows that
any curve
$T^n\gamma_x$ that is
not cut by a singularity or the boundary of a homogeneity strip must
necessarily
intersect $W^2$, except possibly if $T^n\gamma_x$ lies near the endpoints
of $W^1$.
By {\bf (H2)}, $T^i\gamma_x \in \widehat\W^u$ for each $i \ge 0$.

Doing this for each connected component of
$W^1 \setminus \Si_{-n}^T$, we subdivide $W^1 \setminus \Si_{-n}^T$ into
a countable
collection of subintervals of points for which $T^n\gamma_x$ intersects
$W^2 \setminus \Si_{-n}^T$ and subintervals for which this is not the
case.
This in turn induces a corresponding partition on
$W^2 \setminus \Si_{-n}^T$.

We denote by $V^\ell_k$ the pieces in $T^{-n}W^\ell$ which are not matched
up by this process
and note that
the images $T^nV^\ell_k$ occur either at the endpoints of $W^\ell$ or
because the
curve $\gamma_x$ has been cut by a singularity or the boundary of a
homogeneity strip.
In both cases, the length of the
curves $T^nV^\ell_k$ can be at most $C_t \ve$ due to the uniform
transversality of
$\Si_{-n}^T$ with $C^s(x)$, of $C^s(x)$ with $C^u(x)$ and of $C^s(x)$ with
the horizontal.

In the remaining
pieces the foliation $\{ T^n\gamma_x \}_{x \in T^{-n}W^1}$ provides a one
to one correspondence
between points in $W^1$ and $W^2$.
We partition these pieces in
such a way that the lengths of their images under $T^{-i}$
are less than $\delta_0$ for each $0 \le i \le n$ and
the pieces are pairwise matched by
the foliation $\{\gamma_x\}$. We call these matched pieces $\widetilde
U^\ell_j$
and note that $T^i \widetilde U^\ell_j \in \G_{n-i}(W^\ell)$ for each $i =
0, 1, \ldots n$.
For convenience, we further trim the $\widetilde U^\ell_j$ to pieces
$U^\ell_j$ so that $U^1_j$ and $U^2_j$ are both defined on the same
arclength
interval $I_j$. The at most two components of $T^n(\widetilde U^\ell_j
\setminus
U^\ell_j)$ have length less than $C_t \ve$ due to the uniform
transversality
of $C^s(x)$ with the vertical direction. We attach these trimmed pieces to
the adjacent $U^\ell_i$ or $V^\ell_k$ as appropriate so as not to create
any
additional components in the partition.

We further relabel any pieces $U^\ell_j$ as $V^\ell_j$
and consider them unmatched if for some $i$, $0 \le i \le n$,
$|T^i\gamma_x| > 2 |T^iU^\ell_j|$.
i.e. we only consider pieces matched if at each intermediate step, the
distance between
them is at most of the same order as their length. We do this in order to
be able to
apply Lemma~\ref{lem:angles} to the matched pieces. Notice that since the
distance between
the curves at each intermediate step is at most $C_t C_e \ve$ and due to
the uniform
contraction of stable curves going forward, we have $|T^nV^\ell_k| \le C_t
C_e^2 \ve$ for
all such pieces considered unmatched by this last criterion.

In this way we write $W^\ell = (\cup_j T^nU^\ell_j) \cup (\cup_k
T^nV^\ell_k)$.
Note that the images $T^nV^\ell_k$ of the unmatched pieces must have
length $\le C_v \ve$
for some uniform constant $C_v$
while the images of the matched pieces
$U^\ell_j$
may be long or short.

Recalling the notation of Section~\ref{norms}, we have arranged a pairing
of the pieces $U^\ell_j$ with the following property:
\begin{align}\label{eq:match}
\text{If}\,\,
U^1_j = G_{U^1_j}(I_j) = \{ (r, \vf_{U^1_j}(r)) : r \in I_j \},
\,\,\,\text{then }\, U^2_j = G_{U^2_j}(I_j) = \{ (r, \vf_{U^2_j}(r)) : r
\in I_j \},
\end{align}
so that the
point $x = (r, \vf_{U^1_j}(r)) \in U^1_j$ can associated with the point
$\bar x = (r, \vf_{U^2_j}(r)) \in U^2_j$ by the vertical
line $\{(r,s)\}_{s\in[-\pi/2, \pi/2]}$, for each $r \in I_j$. In addition,
the $U^\ell_j$ satisfy
the assumptions of Lemma~\ref{lem:angles}.

Given $\psi_\ell$ on $W^\ell$ with $|\psi_\ell|_{\C^\alpha(W^\ell)} \leq 1$ and
$d_\beta(\psi_1, \psi_2) \leq \ve$,
with the above construction we must estimate
\begin{align}
\label{eq:unstable split}
& \left|\int_{W^1} \Lp_{T,g}^nh \, \psi_1 \, dm_W - \int_{W^2} \Lp_{T,g}^nh \,
\psi_2 \, dm_W \right|
\; \leq \; \sum_{\ell,k} \left|\int_{V^\ell_k} h e^{S_n \hat g}
J_{V^\ell_k}T^n \psi_\ell \circ T^n \, dm_W \right|\nonumber\\
& + \sum_j \left| \int_{U^1_j} h e^{S_n \hat g} J_{U^1_j}T^n \psi_1\circ
T^n \, dm_W
- \int_{U^2_j} h e^{S_n \hat g} J_{U^2_j}T^n \psi_2\circ T^n \, dm_W
\right|
\end{align}
First we estimate the unmatched pieces $V^\ell_k$ 
using the strong stable norm. Note that
by \eqref{eq:C1 C0}, $|\psi_\ell \circ T^n|_{\C^\beta(V^\ell_k)}
\leq C_e |\psi_\ell|_{\C^\alpha(W^\ell)} \leq C_e$.
We estimate as
in Section~\ref{stable norm}, using the fact that $|T^nV^\ell_k| \le C_v
\ve$, as noted above,
\begin{equation}
\label{eq:first unstable}
\begin{split}
\sum_{\ell,k} \Big| \int_{V^\ell_k}h e^{S_n \hat g}
J_{V^\ell_k}T^n\psi_\ell\circ T^n \, & dm_W \Big|
\leq
C_e \sum_{\ell,k} \|h\|_s |V^\ell_k|^p |e^{S_n \hat g}|_{\C^\beta(V^\ell,k)}
|J_{V^\ell_k}T^n|_{\C^\beta(V^\ell,k)} \\
& \leq C_e (1+C_d)^2 C_g |e^{S_n\hg}|_\infty
\| h \|_s \sum_{\ell,k} |V^\ell_k|^p
|J_{V^\ell_k}T^n|_{\C^0(V^\ell_k)} \\
& \leq C' \ve^p C_g |e^{S_n\hg}|_\infty \|h\|_s \sum_{\ell,k}
|J_{V^\ell_k}T^n|_{\C^0(V^\ell_k)}^{1-p}
\le 2 C' \ve^p C_g |e^{S_n\hg}|_\infty \|h\|_s C_3^n ,
\end{split}
\end{equation}
with $C' = C_e (1+C_d)^3 C_v^p$,
where we have applied Lemma~\ref{lem:growth}(d)
with $\varsigma = 1-p > 1/2$
since there are at most two $V^\ell_k$ corresponding to
each element $W^{\ell, n}_i \in \G_n(W^\ell)$ as defined in
Section~\ref{preliminary} and
$|J_{V^\ell_k}T^n|_{\C^0(V^\ell_k)} \leq
|J_{W^{\ell,n}_i}T^n|_{\C^0(W^{\ell,n}_i)}$
whenever $V^\ell_k \subseteq W^{\ell,n}_i$.

Next, we must estimate
\[
\sum_j\left|\int_{U^1_j}he^{S_n \hat g} J_{U^1_j}T^n \, \psi_1\circ T^n \,
dm_W -
\int_{U^2_j}h e^{S_n \hat g} _{U^2_j}T^n \, \psi_2\circ T^n \, dm_W
\right| .
\]
 We fix $j$ and estimate the difference.
Define
\[
\phi_j = (e^{S_n \hat g} J_{U^1_j}T^n \, \psi_1 \circ T^n) \circ G_{U^1_j}
\circ G_{U^2_j}^{-1} .
\]
The function $\phi_j$ is well-defined on $U^2_j$ and we can write,
\begin{equation}
\label{eq:stepone}
\begin{split}
&\left|\int_{U^1_j}he^{S_n \hat g} J_{U^1_j}T^n \, \psi_1\circ T^n-
\int_{U^2_j}h e^{S_n \hat g} J_{U^2_j}T^n \, \psi_2\circ T^n\right|\\
&\leq \left|\int_{U^1_j}h e^{S_n \hat g} J_{U^1_j}T^n \, \psi_1\circ T^n-
\int_{U^2_j}h \,\phi_j \right|
+\left|\int_{U^2_j}h (\phi_j - e^{S_n \hat g} J_{U^2_j}T^n \, \psi_2\circ
T^n)\right| .
\end{split}
\end{equation}

We estimate the first term on the right hand side of~\eqref{eq:stepone}
using the
strong unstable norm.
Using \eqref{eq:holder c0} and \eqref{eq:C1 C0},
\begin{equation}
\label{eq:c1-unst 1}
| e^{S_n \hat g} J_{U^1_j}T^n\cdot \psi_1 \circ
T^n|_{\C^\alpha(U^1_j)}
\le C_e (1+C_d)^2 C_g |e^{S_n\hg}|_\infty
|J_{U^1_j}T^n|_{\C^0(U^1_j)}.
\end{equation}
Notice that
\begin{equation}
\label{eq:graph bound}
|G_{U^1_j} \circ G^{-1}_{U^2_j}|_{\C^1(U^2_j)}
\le \sup_{r \in U^2_j}
\frac{\sqrt{1 + (d\vf_{U^1_j}/dr)^2}}{\sqrt{1 + (d\vf_{U^2_j}/dr)^2}} \le
\sqrt{1 + \Gamma^2} =: C_{a},
\end{equation}
where $\Gamma$ is the maximum slope of curves in $\W^s$
given by {\bf (H1)}.  Using this, we estimate as in \eqref{eq:c1-unst 1},
\[
|\phi_j|_{\C^\alpha(U^2_j)}
\le C_{a} C_e (1+C_d)^2 C_g |e^{S_n\hg}|_\infty
|J_{U^1_j}T^n|_{\C^0(U^1_j)}.
\]
By the definition of $\phi_j$ and $d_\beta(\cdot, \cdot)$,
\[
d_\beta(e^{S_n \hat g} J_{U^1_j}T^n\psi_1\circ T^n, \phi_j)
= \left| \left[ e^{S_n \hat g} J_{U^1_j}T^n\psi_1\circ T^n \right] \circ
G_{U^1_j}
  - \phi_j \circ G_{U^2_j} \right|_{\C^\beta(I_j)} \; = \; 0 .
\]

By Lemma~\ref{lem:angles}(a), we have
$d_{\W^s}(U^1_j,U^2_j)\leq C_0\Lambda^{-n} \ve =: \ve_1$.
In view of \eqref{eq:c1-unst 1} and following, we renormalize the test
functions by
$R_j = C_7 C_g |e^{S_n\hg}|_\infty |J_{U^1_j}T^n|_{\C^0(U^1_j)}$
where $C_7 = C_{a} C_e (1+ C_d)^2$.
Then we apply the definition of the strong unstable norm with
$\ve_1$ in place of $\ve$.      Thus,
\begin{equation}
\label{eq:second unstable}
\sum_j \left|\int_{U^1_j}h e^{S_n \hat g} J_{U^1_j}T^n \, \psi_1\circ T^n
-
\int_{U^2_j} h  \, \phi_j \; \right|
\leq C_7 C_0^\gamma \ve^\gamma \Lambda^{-\gamma n} C_g |e^{S_n\hg}|_\infty 
\|h\|_u \sum_j |J_{U^1_j}T^n|_{\C^0(U^1_j)}
\end{equation}
where the sum is $\le C_2$ by Lemma~\ref{lem:growth}(b) since there is at
most
one matched piece $U^1_j$ corresponding to each element
$W^{1,n}_i \in \G_n(W^1)$ and $|J_{U^1_j}T^n|_{\C^0(U^1_j)} \le
|J_{W^{1,n}_i}T^n|_{\C^0(W^{1,n}_i)}$
whenever $U^1_j \subseteq W^{1,n}_i$.

It remains to estimate the second term in \eqref{eq:stepone} using the
strong stable norm.
\begin{equation}
\label{eq:unstable strong}
\left|\int_{U^2_j}h(\phi_j - e^{S_n \hat g} J_{U^2_j}T^n \psi_2 \circ T^n)
\right|
\leq \;  \|h\|_s |U^2_j|^p
\left|\phi_j - e^{S_n \hat g} J_{U^2_j}T^n \psi_2 \circ
T^n\right|_{\C^\beta(U^2_j)} .
\end{equation}
In order to estimate the $\C^\beta$-norm of the function in \eqref{eq:unstable
strong}, we split
it up into two differences. Since $|G_{U^\ell_j}|_{\C^1} \le C_{a}$ and
$|G_{U^\ell_j}^{-1}|_{\C^1} \le 1$, $\ell = 1,2$,
we write
\begin{equation}
\label{eq:diff}
\begin{split}
& | \phi_j - (e^{S_n \hat g} J_{U^2_j}T^n)\cdot \psi_2 \circ
T^n|_{\C^\beta(U^2_j)} \\
\leq& \; \left |\left[ (e^{S_n \hat g} J_{U^1_j}T^n)\cdot\psi_1 \circ
T^n\right]\circ G_{U^1_j} - \left[(e^{S_n \hat g} J_{U^2_j}T^n)\cdot\psi_2\circ T^n\right] \circ G_{U^2_j}\right|_{\C^\beta(I_j)}\\
\leq& \; \left | (e^{S_n \hat g} J_{U^1_j}T^n)\circ G_{U^1_j} \left[
\psi_1
\circ T^n\circ G_{U^1_j} -\psi_2 \circ T^n\circ
G_{U^2_j}\right]\right|_{\C^\beta(I_j)}\\
&+ \left|\left[(e^{S_n \hat g} J_{U^1_j}T^n) \circ
G_{U^1_j}-(e^{S_n \hat g} J_{U^2_j}T^n) \circ G_{U^2_j}\right]\psi_2\circ
T^n\circ G_{U^2_j}\right|_{\C^\beta(I_j)}\\
\leq& \; C_{a} (1+C_d)^2 C_g | e^{S_n \hat g} J_{U^1_j}T^n|_{\C^0(U^1_j)}
\left|\psi_1 \circ T^n\circ
G_{U^1_j} -\psi_2 \circ T^n\circ G_{U^2_j}\right|_{\C^q(I_j)}\\
&+ C_{a} C_e \left|(e^{S_n \hat g} J_{U^1_j}T^n) \circ
G_{U^1_j}-(e^{S_n \hat g} J_{U^2_j}T^n) \circ
G_{U^2_j}\right|_{\C^\beta(I_j)}
\end{split}
\end{equation}
To bound the two differences above, we need the following lemma, which was
proved in \cite{DZ2} Lemma 4.2. The only difference is the factor $e^{S_n
\hat g}$ which does not play any significant role in the proof, so we omit the proof
here.

\begin{lemma}
\label{lem:test}
There exist constants $C_8, C_9>0$, depending only on {\bf (H1)}-{\bf
(H5)}, such that,
\begin{itemize}
  \item[(a)] $\displaystyle
|(e^{S_n \hat g}  J_{U^1_j}T^n)\circ G_{U^1_j}
-e^{S_n \hat g} J_{U^2_j}T^n)\circ G_{U^2_j}|_{\C^\beta(I_j)}\leq C_8 C_g | e^{S_n
\hat g}
J_{U^2_j}T^n|_{C^0(U^2_j)} \ve^{1/3-\beta};$
  \item[(b)] $\displaystyle
|\psi_1 \circ T^n \circ G_{U^1_j} - \psi_2 \circ T^n \circ G_{U^2_j}
|_{\C^q(I_{r_j})}
\le C_9  C_g \ve^{\alpha-\beta} .$
\end{itemize}
\end{lemma}

It follows from Lemma~\ref{lem:test}(a) that
\[
|e^{S_n \hat g} J_{U^1_j}T^n|_{\C^0(U^1_j)}
 \le (1+C_8C_g \ve^{1/3-\beta}) |e^{S_n \hat g} J_{U^2_j}T^n|_{\C^0(U^2_j)}
\]
which we will use to simplify \eqref{eq:diff}.
Starting from \eqref{eq:unstable strong}, we apply Lemma~\ref{lem:test} to
\eqref{eq:diff}
 to obtain,
\begin{equation}
\label{eq:unstable three}
\begin{split}
& \sum_j \Big| \int_{U^2_j} h(\phi_j - e^{S_n \hat g} J_{U^2_j}T^n \psi_2
\circ T^n ) \, dm_W \Big| \\
& \le \bar C C_g \|h\|_s \sum_j |U^2_j|^p
|e^{S_n \hat g}  J_{U^2_j}T^n|_{\C^0(U^2_j)} \, \ve^{\alpha -\beta}
\le \bar C C_g |e^{S_n\hg}|_\infty \|h\|_s \ve^{\alpha-\beta} \sum_j
|J_{U^2_j}T^n|_{\C^0(U^2_j)},
\end{split}
\end{equation}
for some uniform constant $\bar C$
where again the sum is finite as in \eqref{eq:second unstable}.
This completes the estimate on the second term in \eqref{eq:stepone}.
Now we use this bound, together with \eqref{eq:first unstable} and
\eqref{eq:second unstable} to estimate \eqref{eq:unstable split}
\begin{equation}
\label{eq:final unstable}
\begin{split}
\left|\int_{W^1} \Lp_{T,g}^nh \, \psi_1 \, dm_W - \int_{W^2} \Lp_{T,g}^nh \,
\psi_2 \, dm_W \right|
& \leq  CC_3^n C_g |e^{S_n\hg}|_\infty \|h\|_s \ve^p + C
\|h\|_u \Lambda^{-\gamma n} C_g |e^{S_n\hg}|_\infty \ve^\gamma \\
&  \qquad + C C_g |e^{S_n\hg}|_\infty \|h\|_s \ve^{\alpha-\beta} ,
\end{split}
\end{equation}
where again $C$ depends only on {\bf (H1)}-{\bf (H5)} through the
estimates above.
Since $\alpha-\beta \ge \gamma$ and $p \ge \gamma$, we divide through by
$\ve^\gamma$ and take
the appropriate suprema to complete the proof of \eqref{eq:unstable norm}.


\end{document}